\documentclass{article}

    \usepackage{amsmath,amsfonts,amsthm,amssymb,amscd,color,xcolor,mathrsfs,verbatim,microtype}
    \usepackage{graphicx,eurosym}
    \usepackage{hyperref}
    \usepackage{mathtools}
    \usepackage{bm}

    \usepackage[applemac]{inputenc}
    
    \usepackage[cyr]{aeguill}

    \colorlet{darkblue}{blue!50!black}
    
    \hypersetup{
        colorlinks,%
        citecolor=blue,%
        filecolor=red,%
        linkcolor=darkblue,%
        urlcolor=blue,%
        pdfnewwindow=true,%
        pdfstartview={FitH}
    }

    \usepackage{graphicx,amscd,mathrsfs,wrapfig,mathrsfs,lipsum}
    \usepackage{eufrak}
    \usepackage{float}
    \usepackage{tikz}
    \usepackage{multicol}
    \usepackage{caption}
    \usetikzlibrary{arrows}
    \usepackage{capt-of}
    
    \colorlet{darkblue}{blue!50!black}
    
    \DeclareMathOperator{\lip}{Lip}

    \DeclareMathOperator{\curl}{curl}

    \newcommand{\R}{{\mathbb R}}

    \newcommand{\I}{{\mathbb I}}
    \newcommand{\E}{{\mathbb E}}
    
    \newcommand{\fd}{{\mathfrak d}}

    \newcommand{\ty}{\infty}

    \newcommand{\fw}{{\mathfrak w}}
    
    \newcommand{\fg}{{\mathfrak g}}

    \newcommand{\QQQ}{{\mathrm Q}}
    \newcommand{\PPPPP}{{\mathrm P}}

    \newcommand{\fu}{{\mathfrak u}}
    \newcommand{\fr}{{\mathfrak r}}

    \newcommand{\BB}{{\cal B}}

    \newcommand{\EE}{{\cal E}}

    \newcommand{\LL}{{\cal L}}

    \newcommand{\dd}{{\textup d}}

    \newcommand*\mcap{\mathbin{\mathpalette\mcapinn\relax}}
    \newcommand*\mcapinn[2]{\vcenter{\hbox{$\mathsurround=0pt
      \ifx\displaystyle#1\textstyle\else#1\fi\bigcap$}}}

    \newcommand*\mcup{\mathbin{\mathpalette\mcupinn\relax}}
    \newcommand*\mcupinn[2]{\vcenter{\hbox{$\mathsurround=0pt
      \ifx\displaystyle#1\textstyle\else#1\fi\bigcup$}}}

    \newcommand{\lspan}{\mathop{\rm span}\nolimits}

    \newcommand{\diver}{\mathop{\rm div}\nolimits}

    \theoremstyle{plain}
    \newtheorem*{maintheorem}{Main Theorem}

    \newtheorem*{lemma*}{Lemma}
    \newtheorem{theorem}{Theorem}[section]
    \newtheorem{lemma}[theorem]{Lemma}
    \newtheorem{proposition}[theorem]{Proposition}
    \newtheorem{corollary}[theorem]{Corollary}
    \theoremstyle{definition}
    \newtheorem{definition}[theorem]{Definition}

    \theoremstyle{remark}

    \numberwithin{equation}{section}
    
\begin{document} 
    \author{Vahagn Nersesyan\,\footnote{NYU-ECNU Institute of Mathematical Sciences at NYU Shanghai, 3663 Zhongshan Road North, Shanghai, 200062, China, e-mail: \href{mailto:vahagn.nersesyan@nyu.edu}{Vahagn.Nersesyan@nyu.edu}}     \and  Meng~Zhao\,\footnote{School of Mathematical Sciences, Shanghai Jiao Tong University, 200240 Shanghai, China; Universit\'e Paris Cit\'e and Sorbonne Universit\'e, CNRS, IMJ-PRG, F-75013 Paris, France, e-mail: \href{mailto:mathematics_zm@sjtu.edu.cn}{mathematics$\_$zm@sjtu.edu.cn}}}

    \title{Polynomial mixing for the white-forced Navier--Stokes system in the whole space}
    \maketitle
    \begin{abstract}
    	   We study the mixing properties of the two-dimensional {damped} Navier--Stokes system on the whole space $\R^2$ perturbed by an additive white-in-time noise. Under a suitable non-degeneracy assumption on the noise, we establish the uniqueness of the stationary measure and prove polynomial mixing in the dual-Lipschitz metric. This provides the first ergodicity and mixing result for the white-forced Navier--Stokes system in an unbounded domain. 
    
    The main difficulty arises from the lack of compactness and spectral properties that are essential in bounded settings. To overcome this difficulty, we use a space-time weight function whose square belongs to the Muckenhoupt~$A_2$ class with a characteristic constant bounded uniformly in time. This allows us to control the behavior of solutions at infinity and to derive a suitable Foia\c{s}--Prodi type stability estimate. In addition, the proof combines the coupling method, weighted growth bounds for trajectories, and an estimate for the Leray projector acting on weighted spaces. 
    
        \medskip
        \noindent
        {\bf AMS subject classifications:} 35Q30, 35R60, 37A25, 37L40, 60H15
        
        \medskip
        \noindent
        {\bf Keywords:} Stochastic Navier--Stokes system, polynomial mixing, coupling method, Foia\c{s}--Prodi estimate, weighted growth estimates 
        
        \end{abstract}
        \tableofcontents
    
        \section{Introduction}

       Over the past two decades, substantial progress has been made in studying the ergodic properties of dissipative PDEs driven by white-in-time noise. Most existing results concern PDEs on bounded domains, where one can take advantage of compactness properties, such as compact Sobolev embeddings, and of spectral properties, such as the discreteness of the spectrum of the Laplacian. Among these equations, the Navier--Stokes (NS) system has been studied most extensively because of its fundamental role in fluid dynamics and turbulence. The~first key results were obtained in~\cite{FM95, EMS01, KS02, Hai02, HM06, Oda06}, while~further developments are presented in the book \cite{KS12} and the surveys \cite{KS17, D13}.
        
     In contrast, the situation in unbounded domains is considerably more complicated, since the compactness and spectral properties mentioned above are no longer available. As a result, even the existence of a stationary measure becomes a non-trivial problem; it was established in~\cite{BL06} for unbounded Poincar\'e-type domains and for the damped NS system in the whole space~$\mathbb{R}^2$, while the questions of uniqueness and mixing remained open (see Open Problems~1 and~2 in~\cite{KS17}).
    
      {The difference between the case of bounded domains and the whole-space setting can be seen more concretely through the determining-modes mechanism. On bounded domains, the Poincar\'e inequality and the discreteness of the spectrum of the Stokes operator yield the classical Foia\c{s}--Prodi estimate: if sufficiently many low modes of two solutions are synchronized, then the difference of the remaining components decays exponentially. These low modes therefore determine the long-time dynamics. When the noise is non-degenerate on the determining modes, the Foia\c{s}--Prodi estimate, combined with a recurrence property and an appropriate coupling construction, leads to uniqueness and mixing. On~$\mathbb{R}^2$, however, the Poincar\'e inequality fails, so the viscous term does not control the kinetic energy, and the Stokes operator has continuous rather than discrete spectrum. Consequently, the classical spectral construction of determining modes is no longer available, and finding a suitable replacement for the Foia\c{s}--Prodi mechanism becomes a central difficulty of the present work.}
    
        \subsection{The main result}
        In this work, we present the first result on ergodicity and mixing for the NS system in unbounded domains under additive white-in-time noise. Specifically, we study the dynamics of the following damped randomly forced NS system:
         \begin{equation}
            \label{INTRO1}
            \begin{cases}
                \partial_t u+(u\cdot \nabla) u-\nu\Delta u+au+\nabla p=h+\eta,\qquad x\in \mathbb{R}^2,\\
                \diver u=0,\\
                u|_{t=0}=u_0,
            \end{cases}
        \end{equation}
        where $a,\nu>0$ are respectively the damping parameter and the kinematic viscosity, $u$ is the velocity field, and $p$ is the pressure. We~consider the system \eqref{INTRO1} in the space of divergence-free vector fields 
        \begin{equation}\label{E:space}
        	H:=\{u\in L^2(\mathbb{R}^2;\R^2): \diver u=0\}
        \end{equation}
     endowed with the inner product $\langle \cdot,\cdot\rangle$ inherited from~$L^2:=L^2(\mathbb{R}^2;\R^2)$ and the corresponding norm~$\|u\|:=\langle u,u\rangle^{1/2}$. The external force consists of a deterministic component $h\in H$ and a stochastic component
    \begin{equation}\label{INTRO3}
        \eta(t):=\partial_t\sum_{j=1}^{\infty}b_j\beta_j(t) e_j,
        \end{equation}
    	where   $\{\beta_j\}_{j=1}^{\infty}$ are independent standard Brownian motions defined on a probability space $(\Omega,\mathscr{F},\mathscr{F}_t,\mathbb{P})$  satisfying the usual conditions (e.g., see Definition~2.25 in~\cite{KS91}),
     	$\{b_j\}_{j=1}^{\infty}$ are real numbers with 
        \begin{equation}
            \label{INTRO4_2}
            \mathcal{B}_0:=\sum_{j=1}^\ty b_j^2<\ty,
        \end{equation}
       and   $\{e_j\}_{j=1}^{\infty}$ is an orthonormal basis in $H$.~Under these assumptions, the equation~\eqref{INTRO1} is globally well-posed in $H$ and defines a Markov process.
    
       For any $m>0$, let $H^m:=H^m(\mathbb{R}^2;\R^2)\cap H$ denote the usual Sobolev space of divergence-free vector fields. { Throughout the paper, we fix an orthonormal basis $\{e_j\}_{j\ge 1}$ in $H$ such that
\begin{equation}\label{INTRO3_1}
  e_j\in H^2,\qquad \varphi e_j\in H^1(\mathbb{R}^2;\R^2)\qquad\text{for all } j\ge1,
\end{equation}
where $\varphi(x):=(|x|^2+1)^{\frac14}$. Such a basis exists: applying the Gram--Schmidt procedure in~$H$ to a sequence in $C_0^\infty(\R^2;\R^2)\cap H$ whose linear span is dense in~$H$, we obtain an orthonormal basis of smooth, compactly supported vector fields, which clearly satisfies \eqref{INTRO3_1}.}

        Our main result is stated as follows.
        \begin{maintheorem}{
    	Assume that $\varphi h\in L^2(\R^2;\R^2)$ and
        \begin{equation}
        \label{INTRO4}\mathcal{B}_{\varphi}:=\sum_{j=1}^{\infty}b_j^2\left(\|\varphi e_j\|^2+\|\varphi \curl e_j\|^2\right)<\infty.
    	\end{equation}
        Then, there is an integer $N_0\ge 1$ depending on $a,\nu,h,\BB_\varphi$ such that, if
        \begin{equation}\label{INTRO8}
            b_j\neq 0 \quad \text{for $1\le j\le N_0$},
    \end{equation}the stochastic NS system \eqref{INTRO1} admits a unique stationary measure $\mu$.
    Moreover, for any $q>1$, there is a constant $C_q>0$ such that 
        \[\left|\mathbb{E}(f(u(t)))-\int_{H}f(u)\mu(\mathrm{d} u)\right|\le C_q\|f\|_{\lip}\left(1+t\right)^{-q}\left(1+ \|u_0\|^2\right), \quad t\ge0,\]
            where $u(t)$ is the solution of \eqref{INTRO1} issued from an arbitrary initial condition $u_0\in H$, and $f:H\to \mathbb{R}$ is an arbitrary bounded Lipschitz-continuous function.}

        \end{maintheorem} 
    {  Let us make a few remarks on this result. First, no restriction is imposed on the size of the parameters~$a$ and $\nu$, and no lower bound on the amplitudes $b_1,\ldots,b_{N_0}$ in terms of $a$ and $\nu$ is required.} 
      
     {The mixing rate is polynomial rather than exponential: the growth bounds for the weighted energies in our proof hold only outside events whose probability decays polynomially in time; see Proposition~\ref{proposition4} and Corollary~\ref{corollary1}.}

        { The damping term $au$ plays a crucial role in the whole-space setting. In~boun\-ded domains, or more generally in Poincar\'e-type domains, the Poincar\'e inequality allows the viscous dissipation $\nu\|\nabla u\|^2$ to control the kinetic energy $\|u\|^2$. This mechanism is unavailable on $\mathbb{R}^2$, where the spectrum of $-\Delta$ is $[0,\infty)$, so there is no spectral gap. The damping term restores the missing coercivity at low frequencies and provides the dissipative estimates on which the proof relies.}

       \subsection{Literature review and the scheme of the proof}\label{S:literature}
      In the literature, only a few works consider the problem of the uniqueness of the stationary measure and mixing for randomly forced PDEs in unbounded domains. {Most of these focus on Burgers-type equations and related directed-polymer models. The works~\cite{HK03,BK10} consider random forcing whose dominant part is localized in a compact region of space, whereas the papers~\cite{Bak13,BCK14,BL19,DGR21} deal with spatially homogeneous forcing, which does not decay at infinity. In all these works, uniqueness is obtained through qualitative mechanisms specific to Burgers-type equations, such as variational or Feynman--Kac representations, comparison and contraction principles, and one-force-one-solution arguments, rather than being deduced from a quantitative mixing-rate estimate. These methods rely strongly on the particular structure of Burgers-type equations and do not extend directly to the NS system.}

 For the NS system with non-homogeneous bounded noise, uniqueness and exponential mixing in unbounded domains were established in \cite{N22} via a controllability approach combined with the asymptotic compactness of the dynamics. When the noise is white-in-time, uniqueness and mixing results have so far been obtained only for PDEs with local nonlinearities. In \cite{NZ24}, the stochastic complex Ginzburg--Landau equation is studied, where the main idea of the proof is to quantitatively describe the spatial decay rate of solutions by introducing a space-time weight function, and using the spatial decay to compensate for the loss of compactness. In the paper \cite{G24}, the ideas of \cite{NZ24} were extended to the case of the Kuramoto--Sivashinsky equation. As for the NS system in unbounded domains driven by white-in-time noise, the existence of a stationary measure was proven in~\cite{BL06}, but until now, no results on uniqueness have been available.

     The proof of the Main Theorem relies on the coupling method, as described in Section~3.1.3 in~\cite{KS12} and~\cite{S08}, combined with weighted estimates developed in~\cite{NZ24}. The main difficulty in extending this approach to the NS system in the whole space $\mathbb{R}^2$ lies in the interplay between the non-local nature of the equation and the asymptotic behavior of solutions at infinity. To~overcome it, we introduce a space-time weight function~$\psi$ {(see \eqref{weight})} whose square belongs to the Muckenhoupt $A_2$ class, with the characteristic constant $[\psi^2(t)]_{A_2}$ uniformly bounded in time. This allows us to derive an appropriate estimate for the operator $\psi \Pi$, where $\Pi$ is the Leray projector, which in turn leads to a Foia\c{s}--Prodi type estimate. The latter, combined with growth estimates for weighted parabolic energies and an application of the Girsanov theorem, enables us to verify the polynomial squeezing property, which plays a central role in the coupling argument.
    
     {Let us explain the reasons behind the choice of the space-time weight $\psi$. First, in order for the weighted estimates to hold for all initial conditions in~$H$, without any additional assumption on their spatial decay, the weight must vanish at the initial time: $\psi(0,x)\equiv 0$. Second, $\psi(t,x)$ must tend to infinity as $\min\{t,|x|\}\to\infty$, so that the control of the spatial tails of solutions becomes progressively stronger with time. Third, the derivatives of $\psi$ must be bounded, in order to control the additional terms arising in the weighted velocity and vorticity estimates. Finally, the uniform-in-time bound on $[\psi^2(t)]_{A_2}$ is needed for the corresponding estimate for the weighted Leray projector $\psi\Pi$. The choice~\eqref{weight} is not unique: the arguments of the paper apply to any sufficiently regular space-time weight with these properties that is dominated by the weight $\varphi$ entering the assumptions of the Main Theorem.}

    {We expect that the methods of this paper can be adapted to certain unbounded Poincar\'e-type domains under the Lions boundary conditions. Such an extension would require, in particular, suitable weighted estimates for the Leray projector, which in that case depends on the domain; these questions are not pursued here. Let us note that the case of no-slip boundary conditions seems out of reach of the present approach: our proof relies on a probabilistic linear growth estimate for the vorticity, which is not available in that case, since the vorticity does not satisfy a boundary condition suitable for energy estimates; see~\cite{KV14} for related difficulties in obtaining moment bounds in strong norms.}  
    
    {An interesting open problem is the extension of our results to the three-dimensional stochastic NS system, possibly within the framework of Markov selections (see \cite{DPD03,FR08}); however, the lack of pathwise uniqueness and of suitable stability estimates poses substantial additional difficulties.} 
    
        \smallskip
        
        The structure of the paper is as follows. In Section~\ref{MR}, we provide a detailed construction of the coupling processes and show how the proof of polynomial mixing reduces to verifying recurrence and squeezing properties.~Section~\ref{GE} is dedicated to deriving a growth estimate for a weighted parabolic energy functional. In Section~\ref{STA}, we establish a Foia\c{s}--Prodi type estimate along with growth estimates for an auxiliary process. The recurrence and polynomial squeezing properties are shown in Section~\ref{proofoftheorem2}. Finally, the Appendix collects several technical results used throughout the paper.

      \subsubsection*{Acknowledgement}

      MZ's research was partially supported by the National Natural Science Foundation of China under Grant Nos.~12171317, 12331008, 12250710674, and 12161141004. This paper was completed during MZ's visit to Universit\'e Paris~Cit\'e. He expresses his sincere gratitude to the China Scholarship Council for its financial support, and to Professors Sergei Kuksin and Armen Shirikyan for their invitation, valuable discussions, and generous hospitality during the visit. VN's research was partially supported by the National Natural Science Foundation of China under Grant No.~12571156 and the NYU Shanghai Boost Fund. The authors are grateful to Tianyi Pan for his careful reading of the manuscript and for valuable suggestions. They also thank the anonymous referees for their thorough review and helpful remarks, which improved the paper.
        
        \subsubsection*{Notation}
     
     Let $H$ be the space of divergence-free vector fields defined by \eqref{E:space}, equipped with the $L^2$-inner product $\langle u,v\rangle$
        and the corresponding norm~$\|u\|:=\langle u,u\rangle^{1/2}$.  For~any~$m>0$, let $H^m:=H^m(\mathbb{R}^2;\R^2)\cap H$ be the Sobolev space of divergence-free vector fields. We shall use the following notation.

       \smallskip
       \noindent
       $C_b(H)-$the space of bounded continuous functions   $f:H\to \mathbb{R}$ with the norm
    	\[\|f\|_{\infty}:=\sup_{u\in H}|f(u)|;\]
    
       \smallskip
       \noindent
       $\lip(H)-$the space of bounded Lipschitz-continuous functions   $f:H\to \mathbb{R}$ with 
    	\[\|f\|_{\lip}:=\|f\|_{\infty}+\sup_{\substack{u,v\in H\\ u\neq v}}\frac{|f(u)-f(v)|}{\|u-v\|};\]

       \smallskip
       \noindent
           $B_H(u,R)-$the open ball in $H$ of radius $R>0$ centered at $u\in H$;

        \smallskip
        \noindent
        $\overline{B}_H(u,R)-$the closure of $B_H(u,R)$;
        
        \smallskip
        \noindent 
     $\mathbb{I}_\Gamma-$the  indicator function of a set $\Gamma\subset H$;
    	
        \smallskip
        \noindent
        $\mathscr{B}(H)-$the Borel $\sigma$-algebra of $H$;
    	
        \smallskip
        \noindent
        $\mathscr{P}(H)-$the set of Borel probability measures on $H$. For given $f\in C_b(H)$ and $\lambda\in\mathscr{P}(H)$, 
        we write
        $$
       (f,\lambda):=\int_H f(u)\lambda(\mathrm{d} u).
        $$
         For $\lambda_1,\lambda_2\in \mathscr{P}(H)$, we set 
        \begin{gather*}
          		\|\lambda_1-\lambda_2\|_{\mathcal{L}}^*:=\sup_{\substack{f\in\lip(H)\\\|f\|_{\lip}\le1}}|(f,\lambda_1)-(f,\lambda_2)|,\\
    	\|\lambda_1-\lambda_2\|_{\textup {var}}:=\sup_{\Gamma\in \mathscr{B}(H)}|\lambda_1(\Gamma)-\lambda_2(\Gamma)|=\frac{1}{2}\sup_{\substack{f\in C_{b}(H)\\\|f\|_{\infty}\le1}}|(f,\lambda_1)-(f,\lambda_2)|.
        \end{gather*}

        The distribution of a random variable $\xi$ is denoted by $\mathscr{D}(\xi)$. For real numbers~$a$ and $b$, we use $a\vee b$ to denote their maximum and $a\wedge b$ for their minimum.  
         We~denote by $C$, $C_a$, $C_\nu$, etc. positive constants whose values are not essential to the analysis. For~brevity, we use the notation $\lesssim$, $\lesssim_a$, $\lesssim_\nu $, etc., to indicate inequalities valid up to an inessential multiplicative constant, such as $C$, $C_a$, or~$C_\nu$.

        \section{Construction of a mixing extension}\label{MR}
      
        Under the conditions mentioned in the previous section, the stochastic NS system~\eqref{INTRO1} is well-posed and defines a Markov family $(u_t,\mathbb{P}_u)$ parameterized by the initial condition $u\in H$. The following standard energy estimate is proved in the Appendix:
    	\begin{align}\label{MR1}
    		\mathbb{E}_u\|u(t)\|^2\le e^{-at}\|u\|^2+C_{a,h,\mathcal{B}_0},\quad t\ge0,
    	\end{align}
    	with $\mathbb{E}_u$ being the expectation with respect to $\mathbb{P}_u$. Let $S_t(u,\cdot)$ be the solution issued from $u\in H$, and define the associated Markov operators as follows:   
    	\[
    	 \mathscr{B}_tf(u):=\int_{H}f(v)P_t(u,\mathrm{d} v),\qquad \mathscr{B}_t:C_b(H)\to C_b(H),
    	\]
        \[
    	 \mathscr{B}^*_t\lambda(\Gamma):=\int_{H}P_t(u,\Gamma)\lambda(\mathrm{d} u),\qquad  \mathscr{B}^*_t:\mathscr{P}(H)\to \mathscr{P}(H), 
    	\] 
        where 
        \[P_t(u,\Gamma):=\mathbb{P}\left\{S_t(u,\cdot)\in\Gamma\right\}\]
        is the transition function. A measure $\mu\in \mathscr{P}(H)$ is called stationary for the family $(u_t,\mathbb{P}_u)$ if $\mathscr{B}_t^{*}\mu=\mu$ for any $t>0$. We now restate the Main Theorem as~follows.
    	\begin{theorem}\label{theorem1}
    		Under the assumptions of the Main Theorem, the family $(u(t),\mathbb{P}_u)$ has a unique stationary measure $\mu\in\mathscr{P}(H)$. Moreover, for any $q>1$, there is a constant $C_q>0$ such that 
    	   \[\|\mathscr{B}_t^*\lambda-\mu\|_{\mathcal{L}}^*\le C_q\left(1+t\right)^{-q}\left(1+\int_{H}\|u\|^2\lambda(\mathrm{d} u)\right)\]
    	for any $t\ge0$ and $\lambda\in\mathscr{P}(H)$.
    	\end{theorem}
        \noindent\textit{Scheme of the proof.} The proof relies on the coupling method, as described in Section~3.1.3 of~\cite{KS12} and \cite{S08}. Let us briefly outline the main ideas. Consider a process $(\bm{u}_t,\mathbb{P}_{\bm{u}})$ in $H\times H$, and let $\pi_1$ and $\pi_2$ be the projections from $H\times H$ to the first and second components.~The process $(\bm{u}_t,\mathbb{P}_{\bm{u}})$ is said to be an extension of~$(u_t,\mathbb{P}_u)$ if for any $\bm{u}=(u,u')\in H\times H$, the laws under~$\mathbb{P}_{\bm{u}}$ of processes $\{\pi_1\bm{u}_t\}_{t\ge0}$ and~$\{\pi_2\bm{u}_t\}_{t\ge0}$ coincide with those of $\{u_t\}_{t\ge0}$ under~$\mathbb{P}_u$ and~$\mathbb{P}_{u'}$, respectively. The key idea of the coupling approach is to construct  an extension $(\bm{u}_t,\mathbb{P}_{\bm{u}})$  
        that possesses recurrence and polynomial squeezing properties, as described in the following theorem.

        \begin{theorem} \label{theorem2}
        Under the assumptions of the Main Theorem, for any $q>1$, there is a time $T>0$, a closed set $B\subset H$, an extension $(\bm{u}_t,\mathbb{P}_{\bm{u}})$ of $(u_t,\mathbb{P}_{u})$, and a stopping time~$\sigma$ such that $(\bm{u}_{kT},\mathbb{P}_{\bm{u}})$ is a Markov process and the following properties hold.  
    	\begin{enumerate}
    	\item {\rm \bf (Recurrence)} There are constants $\delta,C>0$ such that 				
        \begin{align}\label{MR2}
\mathbb{E}_{\bm{u}}\exp\left(\delta \tau_B\right)\le C\left(1+\|u\|^2+\|u'\|^2\right)
        \end{align}
    	for any $\bm{u}\in H\times H$, where  
        \begin{align}
            \label{MR2_1}
        \tau_{B}:=\inf\left\{ k\ge0 :\bm{u}_{ k T}\in B\times B\right\}.
        \end{align}
        \item {\rm \bf(Polynomial squeezing)} There are constants $\delta_1,c>0$, which are independent of $q$, such that
        \begin{align}\label{MR3}
        &\|\tilde{u}(t)-\tilde{u}'(t)\|^2\le C_q e^{-c t} \|u-u'\|^2\qquad \mbox{for $0\le t\le \sigma$},\\ 
        \label{MR4}&\mathbb{P}_{\bm{u}}\{{\sigma}=\infty\}\ge \delta_1,\\
        \label{MR5}&\mathbb{E}_{\bm{u}}\left(\mathbb{I}_{\{{\sigma}<\infty\}}\sigma^q\right)\le C_q,\\
        \label{MR6}&\mathbb{E}_{\bm{u}}\left(\mathbb{I}_{\{{\sigma}<\infty\}}\left(\|\tilde{u}({\sigma})\|^{2q}+\|\tilde{u}'({\sigma})\|^{2q}\right)\right)\le C_q
    	\end{align}
        for any $\bm{u}\in B\times B$, where $\tilde{u}_t=\pi_1\bm{u}_t$ and $\tilde{u}'_t=\pi_2\bm{u}_t$. 
        \end{enumerate}
    	\end{theorem}
  An extension is said to be mixing if it satisfies the properties in the above theorem. This is a polynomial version of the notion considered in Section~3.1.3 of~\cite{KS12}: the recurrence property and the bounds \eqref{MR3}, \eqref{MR4}, and  \eqref{MR6} are the same, while in the exponential squeezing property of \cite{KS12}, the estimate \eqref{MR5} is replaced by a bound on the exponential moment. By Theorem~3.1.7 in~\cite{KS12}, the existence of an exponentially mixing extension implies the uniqueness of the stationary measure and exponential mixing for the semigroup $\mathscr{B}^*_{kT}$; under the polynomial condition \eqref{MR5}, the same argument yields polynomial mixing, see Theorem~1.2 in~\cite{G24}. This, in turn, implies polynomial mixing for $\mathscr{B}^*_t$ by virtue of~\eqref{MR1}.

      \smallskip
        \noindent {\it Construction}. Let us outline the construction of a mixing extension $(\bm{u}_t,\mathbb{P}_{\bm{u}})$, following the strategies developed in~\cite{M14} and~\cite{NZ24}. For any $u,u'\in H$, let~$\{u(t)\}_{t\ge 0}$ and $\{u'(t)\}_{t\ge0}$ be the solutions of \eqref{INTRO1} with initial conditions~$u$ and~$u'$. Let the process
     $\{v(t)\}_{t\ge0}$ be the solution of the following auxiliary problem:
        \begin{align}
        \label{MR7}
        \begin{cases}
    	\partial_t v+\Pi(v\cdot \nabla) v-\nu\Delta v+av\\\qquad\qquad\qquad+\mathrm{P}_N[\Pi(u\cdot \nabla) u-\Pi(v\cdot \nabla) v-\nu\Delta(u-v)]=h+\eta,\\
    	v|_{t=0}=u',
    	\end{cases}
        \end{align}
        where $\mathrm{P}_N$ is the orthogonal projection in $H$ onto the space $\lspan\{e_1,\ldots,e_N\}$, with $N\ge1$ to be specified later, $\Pi$ is the Leray projector, and we used the fact that $\Pi\Delta=\Delta$ on $\mathbb{R}^2$.
    
         Let us fix a time step $T>0$ to be specified later, and denote by $\lambda_T(u,u')$ and~$\lambda'_T(u,u')$ the distributions of processes $\{v(t)\}_{t\in[0,T]}$ and $\{u'(t)\}_{t\in[0,T]}$. By~Theorem~1.2.28 in \cite{KS12}, there is a maximal coupling $(\mathcal{V}_T(u,u'),\mathcal{V}_T'(u,u'))$ for the pair of measures $(\lambda_T(u,u'),\lambda'_T(u,u'))$ defined on some probability space $(\tilde{\Omega}, \tilde{\mathscr{F}}, \tilde{\mathbb{P}})$ and depending measurably on $(u,u')$. That is, we have (see Definition~1.2.21 in~\cite{KS12})
        \[\tilde{\mathbb{P}}\{\mathcal{V}_T(u,u')\neq \mathcal{V}_T'(u,u')\}=\|\lambda_T(u,u')-\lambda'_T(u,u')\|_{\textup {var}},\]
        and conditioned on the event $\{\mathcal{V}_T(u,u')\neq \mathcal{V}_T'(u,u')\}$, the random variables $\mathcal{V}_T(u,u')$ and $\mathcal{V}_T'(u,u')$  are independent. Let $\{\tilde{v}(t)\}_{t\in[0,T]}$ and $\{\tilde{u}'(t)\}_{t\in[0,T]}$ be the flows of this maximal coupling. Then, the process $\{\tilde{v}(t)\}_{t\in [0,T]}$ is the solution~of 
        \begin{align}
        \label{MR8}
        \begin{cases}
    	\partial_t \tilde{v}+\Pi(\tilde{v}\cdot \nabla )\tilde{v}-\nu\Delta \tilde{v}+a\tilde{v}+\mathrm{P}_N[\nu\Delta \tilde{v}-\Pi(\tilde{v}\cdot \nabla )\tilde{v}]=h+\Lambda,\\
    	\tilde{v}|_{t=0}=u',
    	\end{cases}
        \end{align}
        where the distribution of the process $\left\{\int_0^{t}\Lambda(s)\mathrm{d} s\right\}_{t\in[0,T]}$ is equal to that of   
    	\[
    	\left\{\int_0^{t}\left(\eta(s)-\mathrm{P}_N \Pi(u\cdot \nabla )u+\nu \mathrm{P}_N\Delta u\right)\mathrm{d} s\right\}_{t\in[0,T]}\]
        and $\eta$ is defined by \eqref{INTRO3}. Let $\{\tilde{u}(t)\}_{t\in [0,T]}$ be the solution of 
    		\begin{align}\label{MR9}\begin{cases}
    				\partial_t \tilde{u}+\Pi(\tilde{u}\cdot \nabla )\tilde{u}-\nu\Delta\tilde{u}+a \tilde{u}+\mathrm{P}_N\left[\nu\Delta\tilde{u}-\Pi(\tilde{u}\cdot \nabla )\tilde{u}\right]=h+\Lambda,\\
    				\tilde{u}|_{t=0}=u.
    		\end{cases}\end{align}
            {We claim that 
            \begin{align}
                \label{MR9-1}\mathscr{D}(\{\tilde{u}(t)\}_{t\in[0,T]})=\mathscr{D}(\{u(t)\}_{t\in[0,T]}).
            \end{align}
            Indeed, by a standard Galerkin construction and uniqueness argument (see, e.g., Chapter~2 in \cite{KS12}), it follows that \eqref{MR9} defines a Borel measurable solution map
            \[\mathcal S_T: H\times C_0([0,T];H^1)\longrightarrow C([0,T];H)\cap L^2(0,T;H^1),\]
            where $C_0([0,T];H^1)$ denotes the space of continuous functions $f:[0,T]\to H^1$ vanishing at $t=0$.  Notice that $u(t)$ solves \eqref{MR9} with $\Lambda$ replaced by 
\[\Lambda_u:=\eta-\mathrm{P}_N \Pi(u\cdot \nabla )u+\nu \mathrm{P}_N\Delta u,\]
and, by construction, the primitives of $\Lambda$ and $\Lambda_u$ have the same distribution in~$C_0([0,T];H^1)$. Consequently, the processes $\{\tilde u(t)\}_{t\in [0,T]}$ and $\{u(t)\}_{t\in[0,T]}$ are the images of identically distributed driving processes under the same solution map $\mathcal{S}_T(u,\cdot)$, so their distributions coincide. This ensures \eqref{MR9-1}.}

        For any $u, u'\in H$,  $\omega\in\tilde\Omega$, and $t\in [0,T]$, let us define  
        \[\mathcal{R}_t(u,u',\omega):=\tilde{u}_t,\quad \mathcal{R}'_t(u,u',\omega):=\tilde{u}_t',\qquad\mathcal{R}''_t(u,u',\omega):=\tilde v_t.\]
          Consider a sequence of independent copies  $\{(\Omega^k,\mathscr{F}^k,\mathbb{P}^k)\}_{k\ge0}$ of the probability space $(\tilde{\Omega}, \tilde{\mathscr{F}},\tilde{\mathbb{P}})$, and let $(\Omega,\mathscr{F},\mathbb{P})$ denote their direct product.   For~any  $u,u'\in H$ and $\omega=(\omega^0,\omega^1,\ldots)\in\Omega$, let $\tilde u_0=u$, $\tilde u_0'=u'$, and 
        \[\tilde u_t(\omega):=\mathcal{R}_s(\tilde{u}_{kT}(\omega),\tilde{u}'_{kT}(\omega),\omega^k), \quad
        \tilde{u}'_t(\omega):=\mathcal{R}_s'(\tilde{u}_{kT}(\omega),\tilde{u}'_{kT}(\omega),\omega^k),\]
        \[\tilde{v}_t(\omega):=\mathcal{R}''_s(\tilde{u}_{kT}(\omega),\tilde{u}'_{kT}(\omega),\omega^k),\]
        where $t=s+kT$, $s\in (0,T]$, and $k\ge0$. Finally,
         we introduce the pair
         \[\bm{u}_t:=(\tilde u_t,\tilde u_t').\] This construction ensures that $(\bm{u}_t,\mathbb{P}_{\bm{u}})$  is an extension of  $(u_t,\mathbb{P}_u)$.    For suitable choices of the parameters $N$ and $T$, we will show in Section \ref{proofoftheorem2} that the process~$(\bm{u}_t,\mathbb{P}_{\bm{u}})$ is a mixing extension for $(u_t,\mathbb{P}_u)$. To prepare for this, in Sections~\ref{GE} and~\ref{STA}, we establish the necessary groundwork by studying some growth estimates for energy functionals and stability properties of the stochastic NS system \eqref{INTRO1}.

        \section{Growth estimates for solutions}\label{GE}
        {Let us define the velocity energy functional
        \begin{equation}
            \label{G2}
            \mathcal{E}^u_{p}(t):=\| u(t)\|^{2p}+\int_0^t\left(\| u(s)\|^{2p-2}\|\nabla u(s)\|^2+\| u(s)\|^{2p}\right)\mathrm{d}s,
        \end{equation}
        the (time-weighted) vorticity energy functional
        \begin{align}\label{GC2-1}
            \tilde{\mathcal{E}}^u_{1,1}(t)&:=(1\wedge t)\|w(t+1)\|^2+\int_0^t\|w(s+1)\|^2_{H^1}\mathrm{d}s,
        \end{align}
        and the space-time weighted energy functionals
        \begin{align}
            \label{GC2-2}\tilde{\mathcal{E}}^u_{\psi}(t):=\| \psi(t)u(t+1)\|^{2}+\int_0^t\left(\|\psi(s)\nabla u(s+1)\|^2+\| \psi(s)u(s+1)\|^{2}\right)\mathrm{d}s,
        \end{align}
        and 
        \begin{align}
            \label{GC2-3}\tilde{\mathcal{E}}_{1,\psi}^u(t):=\|\psi(t)w(t+1)\|^2+\int_0^t\left(\|\psi(s) \nabla w(s+1)\|^2+\|\psi(s) w(s+1)\|^2\right)\mathrm{d}s,
        \end{align}
        where $u(t)$ is the solution of the stochastic NS system \eqref{INTRO1} issued from $u_0$, $w:=\curl u$ is the corresponding vorticity, and $\psi$ is a space-time weight function given by
       \begin{equation}\label{weight}
           \psi(t,x):=\varphi(x)\left(1-\exp\left(-\frac{t}{\varphi(x)}\right)\right),\qquad\varphi(x):=(|x|^2+1)^{\frac{1}{4}}.
       \end{equation}		
        The main result of this section is Proposition \ref{proposition4}, which establishes a growth estimate for the weighted energy functional 
        \begin{align}\label{G1}
            \mathcal{E}^u_{\psi}(t):=\begin{cases}\mathcal{E}^u_{1}(t)+\mathcal{E}^u_{3}(t),\qquad\qquad\qquad\qquad\qquad\qquad\qquad\qquad\quad 0\le t<1,\\\mathcal{E}^u_{1}(t)+\mathcal{E}^u_{3}(t)+\tilde{\mathcal{E}}^u_{1,1}(t-1)+\tilde{\mathcal{E}}^u_{\psi}(t-1)+\tilde{\mathcal{E}}_{1,\psi}^u(t-1),\qquad t\ge 1.
            \end{cases}
        \end{align}}
        The weight function $\psi$ possesses the following properties, which will be used throughout the paper:
    		\begin{itemize}
    			\item[\hypertarget{(i)}{\bf(i)}]    $0<\psi(t,x)<\varphi(x)$ and $\psi(0,x)=0$ for any $t> 0$ and $x\in\R^2$;
    			\item[\hypertarget{(ii)}{\bf(ii)}]  the partial derivatives of $\psi$ of order $\ge 1$ are bounded;
    			\item[\hypertarget{(iii)}{\bf(iii)}] as 
    			$t,|x|\to+\infty$, there holds $\psi(t,x)\to +\infty$;
    			\item[\hypertarget{(iv)}{\bf(iv)}]	for every $t\ge1$, $\psi^2(t)$ belongs to the Muckenhoupt $A_2$ class and \[\sup_{t\ge 1} [\psi^2 (t)]_{A_2}<\infty.\]
    		\end{itemize} 
    		 The proof of the properties \hyperlink{(i)}{\rm(i)}--\hyperlink{(iii)}{\rm(iii)} is straightforward, cf. Section~2.1 in~\cite{NZ24}, and \hyperlink{(iv)}{\rm(iv)} is established in 	Lemma~\ref{lemmab1}. To~estimate the energy $\mathcal{E}^u_{\psi}(t)$, we decompose it into several components, each of which is estimated in the following four subsections. In what follows, we always assume that the hypotheses of the Main~Theorem are satisfied.

        \subsection{\texorpdfstring{$L^2$}{L2}-estimate for the velocity field}\label{GEV}

      Let us begin with the velocity energy functional $\EE^u_p(t)$ defined in \eqref{G2}.
        
        \begin{proposition}\label{proposition1}
                For any $p\ge1$, there exist constants $\kappa_{p}$ and $\mathcal{C}_p$  depending on $a,\nu,h,\mathcal{B}_\varphi,p$ such that the following estimate holds: 
            \begin{align*}\mathbb{P}\left\{\sup_{t\ge 0}\left(\mathcal{E}_{p}^{u}(t)-\kappa_pt-\mathcal{C}_p\|u_0\|^{2p}\right)\ge \rho\right\}\le C_{a,\nu,h,\mathcal{B}_\varphi,p,q}\frac{\mathbb{E}\|u_0\|^{(2p-1)q}+1}{\rho^{\frac{q}{2}-1}}
        \end{align*}        for any $q,\rho>2$.
        \end{proposition}
        \begin{proof}
        By the It\^o formula and the cancellation property of the convection term, we have
        \begin{align*}
            \mathrm{d}\|u\|^{2p}&=p\|u\|^{2(p-1)}\left(2\langle u,\nu\Delta u-au+h\rangle\mathrm{d}t+2\sum_{j=1}^{\infty}b_j\langle u,e_j\rangle\mathrm{d}\beta_j+\mathcal{B}_0\mathrm{d}t\right)\notag\\&\quad+2p(p-1)\|u\|^{2(p-2)}\sum_{j=1}^{\infty}b_j^2\langle u,e_j\rangle^2\mathrm{d}t.
        \end{align*}
        Integrating by parts, we derive 
        \begin{align}
            \label{G5}
            \|u(t)\|^{2p}&+\int_0^t\left(2p\nu\|\nabla u\|^2\|u\|^{2(p-1)}+2ap\|u\|^{2p}\right)\mathrm{d}s\notag\\&\quad\le \|u_0\|^{2p}+\int_0^tp\|u\|^{2(p-1)}\left(2\langle u,h\rangle+(2p-1)\mathcal{B}_\varphi\right)\mathrm{d}s+M_p(t),
        \end{align}
        where
        \[M_p(t):=2p\int_0^t\|u\|^{2(p-1)}\sum_{j=1}^{\infty}b_j\langle u,e_j\rangle\mathrm{d}\beta_j.\]
        Then, by the Young inequality, 
        \begin{align}
            \label{G6}
            \mathcal{E}_{p}^{u}(t)-C_{a,\nu,p}\|u_0\|^{2p}\le C_{a,\nu,h,\mathcal{B}_\varphi,p}t+C_{a,\nu,p}M_p(t),
        \end{align}
        which implies that 
       $$
            \mathcal{E}_{p}^{u}(t)-\kappa_pt-C_{a,\nu,p}\|u_0\|^{2p}\le C_{a,\nu,p}\left(M_p(t)-t\right)
     $$
        for some positive constant $\kappa_p$. Therefore,
        \begin{align*}
            &\left\{\sup_{t\ge 0}\left(\mathcal{E}_{p}^{u}(t)-\kappa_pt-C_{a,\nu,p}\|u_0\|^{2p}\right)\ge \rho\right\}\subset \left\{\sup_{t\ge 0}\left(M_p(t)-t\right)\ge \frac{\rho}{C_{a,\nu,p}}\right\}\notag\\&\qquad\qquad\qquad\qquad\qquad\qquad\qquad\qquad =\bigcup_{n=0}^{\infty}\left\{\sup_{t\in [n,n+1)}\left(M_p(t)-t\right)\ge \frac{\rho}{C_{a,\nu,p}}\right\}\notag\\&\qquad\qquad\qquad\qquad\qquad\qquad\qquad\qquad\subset \bigcup_{n=0}^{\infty}\left\{\sup_{t< n+1}|M_{p}(t)|\ge \frac{\rho}{C_{a,\nu,p}}+n\right\}.
        \end{align*}
        An application of the Chebyshev and Burkholder--Davis--Gundy inequalities further yields
        \begin{align}\label{G8}
            \mathbb{P}&\left\{\sup_{t\ge 0}\left(\mathcal{E}_{p}^{u}(t)-\kappa_pt-C_{a,\nu,p}\|u_0\|^{2p}\right)\ge \rho\right\}\notag\\&\qquad\qquad\qquad\qquad\qquad\qquad\le \sum_{n=0}^{\infty}\mathbb{P}\left\{\sup_{t< n+1}|M_{p}(t)|\ge \frac{\rho}{C_{a,\nu,p}}+n\right\}\notag\\&\qquad\qquad\qquad\qquad\qquad\qquad\le C_{a,\nu,p,q}\sum_{n=0}^{\infty}\frac{\mathbb{E}\sup_{t< n+1}|M_{p}(t)|^{q}}{(\rho+n)^{q}}\notag\\&\qquad\qquad\qquad\qquad\qquad\qquad\le C_{a,\nu,p,q}\sum_{n=0}^{\infty}\frac{\mathbb{E}\left(\langle M_{p}\rangle(n+1)\right)^{\frac{q}{2}}}{(\rho+n)^{q}},
        \end{align}
        where 
        \begin{align*}
            \langle M_{p}\rangle(t):=4p^2\int_0^t\|u\|^{4(p-1)}\sum_{j=1}^{\infty}b_j^2\langle u,e_j\rangle^2\mathrm{d}s
        \end{align*}
          is the quadratic variation of the martingale $M_p$.
        Applying the H\"older inequality, we find that
        \begin{align}
            \label{G9}
            \langle M_{p}\rangle^{\frac{q}{2}}(t)\le (2p)^{q}\left(\int_0^t\mathcal{B}_0\|u\|^{4p-2}\mathrm{d}s\right)^{\frac{q}{2}}\le (2p)^{q}\mathcal{B}_\varphi^{\frac{q}{2}}t^{\frac{q}{2}-1}\int_0^t\|u\|^{(2p-1)q}\mathrm{d}s.
        \end{align}
        To bound the term on the right-hand side of this inequality, we take the expectation in \eqref{G6} with the parameter $p$ replaced by $(p-\frac{1}{2})q$, yielding  
        \begin{align}\label{G10}
        \mathbb{E}\int_0^t\|u\|^{(2p-1)q}\mathrm{d}s\le C_{a,\nu,p,q}\mathbb{E}\|u_0\|^{(2p-1)q}+C_{a,\nu,h,\mathcal{B}_\varphi,p,q}t.
        \end{align}
        Combining the estimates \eqref{G8}--\eqref{G10}, we arrive at
        \begin{align*}
            &\mathbb{P}\left\{\sup_{t\ge 0}\left(\mathcal{E}_{p}^{u}(t)-\kappa_pt-C_{a,\nu,p}\|u_0\|^{2p}\right)\ge \rho\right\}\notag\\&\qquad\qquad\qquad\qquad\qquad\qquad\qquad\le C_{a,\nu,h,\mathcal{B}_\varphi,p,q}\sum_{n=0}^{\infty}\frac{\mathbb{E}\|u_0\|^{(2p-1)q}+1}{(\rho+n)^{\frac{q}{2}}}.
        \end{align*}
        As
        \begin{align*}
            \sum_{n=0}^{\infty}\frac{1}{(\rho+n)^{\frac{q}{2}}}\le \int_{\rho-1}^{\infty} \frac{1}{x^{\frac{q}{2}}}\mathrm{d}x=\frac{2}{q-2}\frac{1}{(\rho-1)^{\frac{q}{2}-1}}\lesssim_q \frac{1}{\rho^{\frac{q}{2}-1}}
        \end{align*}
        for $q>2$, the proof is complete.
        \end{proof}

        \subsection{\texorpdfstring{$L^2$}{L2}-estimate for the vorticity}\label{GEVorticity}
        To control the growth of the vorticity energy functional $\tilde\EE^u_{1,1}$ defined in \eqref{GC2-1}, it suffices to consider
        \begin{align*}
            \mathcal{E}^u_{1,p}(t):=\| w(t+1)\|^{2p}+\int_0^t\left(\| w(s+1)\|^{2p-2}\|\nabla w(s+1)\|^2+\| w(s+1)\|^{2p}\right)\mathrm{d}s,
        \end{align*}
        as one has $\tilde\EE^u_{1,1}(t)\le \EE^u_{1,1}(t)$ for any $t\ge0$. Let us note that
        \begin{align}
            \label{G10-1}\mathcal{B}_{1}:=\sum_{j=1}^{\infty}b_j^2\| e_j\|_{H^1}^2\le \BB_\varphi<\infty,
        \end{align}
         due to the bound $\varphi\ge1$, the identity $\|\nabla e_j\|=\|\curl e_j\|$ for divergence-free vector fields, and the condition \eqref{INTRO4}.
        \begin{proposition}\label{proposition1_1}
            For any $p\ge1$, there exist constants $\kappa_{1,p}$ and $\mathcal{C}_{1,p}$ depending on $a,\nu,h,\mathcal{B}_\varphi,p$ such that         \begin{align*}
            \mathbb{P}&\left\{\sup_{t\ge 0}\left(\mathcal{E}_{1,p}^{u}(t)-\kappa_{1,p}t-\mathcal{C}_{1,p}\|w(1)\|^{2p}\right)\ge \rho\right\}\le C_{a,\nu,h,\mathcal{B}_\varphi,p,q}\frac{\mathbb{E}\|u_0\|^{4(pq+2)}+1}{\rho^{\frac{q}{2}-1}}
        \end{align*} for any $q,\rho>2$.
        \end{proposition}
        \begin{proof} The proof is similar to that of Proposition~\ref{proposition1}. 
        The equation for $w$ is given by 
        \begin{align}\label{G11_0}
            \partial_t w+u\cdot \nabla w-\nu\Delta w+aw=\curl h+\curl \eta.
        \end{align}
        By the It\^o formula and the cancellation property of the convection term, 
        \begin{align*}
            \mathrm{d}\|w\|^{2p}&
            =2p\|w\|^{2(p-1)}\langle w,\nu\Delta w-aw+\curl h\rangle\mathrm{d}t\notag\\&\quad +p\|w\|^{2(p-1)}\left(2\sum_{j=1}^{\infty}b_j\langle w,\curl e_j\rangle\mathrm{d}\beta_j+\sum_{j= 1}^{\infty}b_j^2\|\curl e_j\|^2\mathrm{d}t\right)\notag\\&\quad +2p(p-1)\|w\|^{2(p-2)}\sum_{j=1}^{\infty}b_j^2\langle w,\curl e_j\rangle^2\mathrm{d}t.
        \end{align*}
        Integrating this equality from $1$ to $t+1$ and using \eqref{G10-1}, we derive 
        \begin{align*}
            \label{G11_2}
            &\|w(t+1)\|^{2p}-\|w(1)\|^{2p}+\int_1^{t+1}\left(2p\nu\|w\|^{2(p-1)}\|\nabla w\|^2+2ap\|w\|^{2p}\right)\mathrm{d}s\notag\\&\quad\qquad  \le \int_1^{t+1}p\|w\|^{2(p-1)}\left(2\langle w,\curl h\rangle+(2p-1)\mathcal{B}_\varphi\right)\mathrm{d}s+M_{1,p}(t),
        \end{align*}
        where 
        \[M_{1,p}(t):=2p\int_1^{t+1}\|w\|^{2(p-1)}\sum_{j=1}^{\infty}b_j\langle w,\curl e_j\rangle\mathrm{d}\beta_j.\]
        Applying the Young inequality, we find a positive constant $\kappa_{1,p}$ depending on $a,\nu,h,\mathcal{B}_\varphi,p$ such that 
        \begin{align}
            \label{G11_3}
            \mathcal{E}_{1,p}^{u}(t)-\kappa_{1,p}t-C_{a,\nu,p}\|w(1)\|^{2p}\le C_{a,\nu,p}\left(M_{1,p}(t)-t\right).
        \end{align}
    Repeating the arguments of the proof of \eqref{G8}, we obtain
        \begin{align}\label{G11_4}
            \mathbb{P}&\left\{\sup_{t\ge 0}\left(\mathcal{E}_{1,p}^{u}(t)-\kappa_{1,p}t-C_{a,\nu,p}\|w(1)\|^{2p}\right)\ge \rho\right\}\notag\\&\qquad\qquad\qquad\qquad\qquad\qquad\le C_{a,\nu,p,q}\sum_{n=0}^{\infty}\frac{\mathbb{E}\left(\langle M_{1,p}\rangle(n+1)\right)^{\frac{q}{2}}}{(\rho+n)^{q}},
        \end{align}
        where 
        \begin{align*}
            \langle M_{1,p}\rangle(t):=4p^2\int_1^{t+1}\|w\|^{4(p-1)}\sum_{j=1}^{\infty}b_j^2\langle w,\curl e_j\rangle^2\mathrm{d}s
        \end{align*}is the quadratic variation of $M_{1,p}(t)$. 
        The H\"older inequality implies that
        \begin{align}
            \label{G11_5}
            \langle M_{1,p}\rangle^{\frac{q}{2}}(t)&\le (2p)^{q}\left(\int_1^{t+1}\mathcal{B}_1\|w\|^{4p-2}\mathrm{d}s\right)^{\frac{q}{2}}\notag\\&\le (2p)^{q}\mathcal{B}_{\varphi}^{\frac{q}{2}}t^{\frac{q}{2}-1}\int_1^{t+1}\|w\|^{(2p-1)q}\mathrm{d}s.
        \end{align}
        From the inequality \eqref{G11_3} with $p$ replaced by $(p-\frac{1}{2})q$ and Lemma \ref{lemmaA1} it follows~that 
            \begin{align}\label{G11_6}
        \mathbb{E}\int_1^t\|w\|^{(2p-1)q}\mathrm{d}s&\le C_{a,\nu,p,q}\mathbb{E}\|w(1)\|^{(2p-1)q}+C_{a,\nu,h,\mathcal{B}_\varphi,p,q}t\notag\\&\le C_{a,\nu,h,\mathcal{B}_\varphi,p,q}\left(\mathbb{E}\|u_0\|^{4(pq+2)}+t\right), \quad t\ge1.
        \end{align}
        Combining the estimates \eqref{G11_4}--\eqref{G11_6}, we derive 
        \begin{align*}
            \mathbb{P}&\left\{\sup_{t\ge 0}\left(\mathcal{E}_{1,p}^{u}(t)-\kappa_{1,p}t-C_{a,\nu,p}\|w(1)\|^{2p}\right)\ge \rho\right\}\notag\\&\qquad\qquad\qquad\qquad\qquad\qquad\le C_{a,\nu,h,\mathcal{B}_\varphi,p,q}\sum_{n=0}^{\infty}\frac{\mathbb{E}\|u_0\|^{4(pq+2)}+1}{(\rho+n)^{\frac{q}{2}}}\notag\\&\qquad\qquad\qquad\qquad\qquad\qquad\le C_{a,\nu,h,\mathcal{B}_\varphi,p,q}\frac{\mathbb{E}\|u_0\|^{4(pq+2)}+1}{\rho^{\frac{q}{2}-1}},
        \end{align*}
        which completes the proof.
        \end{proof}

        \subsection{Weighted estimate for the velocity field}\label{GEWV}
       Next, we establish a growth estimate for the space-time weighted energy functional $\tilde\EE^u_\psi$ defined in \eqref{GC2-2}.
         \begin{proposition}\label{proposition2}There exist constants $\tilde{\kappa},\tilde{\gamma},\tilde{\mathcal{C}}$ depending on $a,\nu,h,\mathcal{B}_{\varphi}$ such that for any $q,\rho>2$,
            \begin{align*}
             \mathbb{P}&\left\{\sup_{t\ge 0}\left(\tilde{\mathcal{E}}^u_{\psi}(t)-\tilde{\kappa}t-\tilde{\mathcal{C}}\left(1+\|u_0\|^6\right)\right)\ge \rho\right\}\\&\quad\qquad\qquad\qquad\qquad\qquad\le C_{a,\nu,h,\mathcal{B}_\varphi,q}\left(\mathbb{E}\|u_0\|^{5q}+1\right)\left(e^{-\tilde\gamma \rho}+\frac{1}{\rho^{\frac{q}{2}-1}}\right).
        \end{align*}
        \end{proposition}
        \begin{proof} To simplify the notation, let us write 
      $\tilde{\psi}(t+1):=\psi(t).$ 
        By the It\^o formula,     
        \begin{align}
            \label{G13}
            \mathrm{d}\|\tilde\psi u\|^2&=2\langle\partial_t\tilde\psi u,\tilde\psi u\rangle \dd t+2\langle \tilde\psi u, \tilde\psi(-(u\cdot \nabla )u+ \nu\Delta u-au+h-\nabla p)\rangle\mathrm{d}t\notag\\&\quad+\mathrm{d}M_{\tilde\psi}(t)+\sum_{j=1}^{\infty}b_j^2\|\tilde\psi e_j\|^2\mathrm{d}t,
        \end{align}
        where
        \[M_{\tilde\psi}(t):=2\int_1^{t+1}\sum_{j=1}^{\infty}b_j\langle \tilde\psi u,\tilde\psi e_j\rangle\mathrm{d}\beta_j.\]
        For the nonlinear convection term, we use the property  \hyperlink{(ii)}{\rm(ii)} of the weight function $\psi$ and the Ladyzhenskaya inequality
        \begin{align}
            \label{lady}
            \|f\|_{L^4}\lesssim \|f\|^{\frac{1}{2}}\|\nabla f\|^{\frac{1}{2}}
        \end{align}
        to derive 
        \begin{align}
            \label{G14}
            -\langle\tilde\psi u,\tilde\psi (u\cdot\nabla) u\rangle&= \langle \nabla \tilde\psi\cdot u,\tilde\psi |u|^2\rangle\notag\\&\le C\|\tilde\psi u\|_{L^4}\|u\|_{L^4}\|u\|\notag\\&\le C\|\tilde\psi u\|^{\frac{1}{2}}\|\tilde\psi \nabla u\|^{\frac{1}{2}}\|u\|^{\frac{3}{2}}\|\nabla u\|^{\frac{1}{2}}+C\|\tilde\psi u\|^{\frac{1}{2}}\|u\|^{2}\|\nabla u\|^{\frac{1}{2}}\notag\\&\le \delta\|\tilde\psi u\|^2+\delta\|\tilde\psi\nabla u\|^2+C\|\nabla u\|^2+ C\left(1+\|u\|^6\right),
        \end{align}
        where $\delta>0$ is a small parameter to be determined. As for the pressure term,  using again the property \hyperlink{(ii)}{\rm(ii)}  of the weight function $\psi$,
        \begin{align}
            \label{G15}
            \langle \tilde\psi u,\tilde\psi \nabla p\rangle=-2\langle \tilde\psi u,\nabla  \tilde\psi p\rangle\le C \|\tilde\psi u\|\| p\|.
        \end{align}
        Moreover, since $p$ satisfies 
        \[-\Delta p=\diver ((u\cdot \nabla )u)=\sum_{k,j=1}^2\partial_k\partial_j(u_ku_j),\]
        we have 
        \[p=\sum_{k,j=1}^2\mathcal{R}_k\mathcal{R}_j(u_ku_j),\]
        where 
        \begin{align}\label{Riesz}
    \mathcal{R}_jf(x):=\mathcal{F}^{-1}\left(-i\frac{\xi_j}{|\xi|}\mathcal{F}f(\xi)\right)(x)    \end{align} 
        is the Riesz transform and $\mathcal{F}$ is the Fourier transform. By the $L^2$-boundedness of $\mathcal{R}_j$ and the Ladyzhenskaya inequality \eqref{lady}, we have 
       $$
            \| p\|\lesssim \sum_{i,j=1}^2\|u_iu_j\|\lesssim  \|u\|_{L^4}^2\lesssim\|u\|\|\nabla u\|.
     $$
        Combining this with \eqref{G15} and the Young inequality, we obtain
        \begin{align}
            \label{G17}
            \langle \tilde\psi u,\tilde\psi \nabla p\rangle&\le C\|\tilde\psi u\|\|u\|\|\nabla u\|\le \delta \|\tilde\psi u\|^2+C\|u\|^2\|\nabla u\|^2.
        \end{align}
        Notice that 
        \[\sum_{j=1}^{\infty}b_j^2\|\tilde\psi(s) e_j\|^2\le \sum_{j=1}^{\infty}b_j^2\|\varphi e_j\|^2 \le \mathcal{B}_{\varphi},\qquad s\ge 1,\]
        due to the property \hyperlink{(i)}{\rm(i)}  of the weight function $\psi$, and the terms involving $\partial_t\psi$ and $\nabla\psi$ in \eqref{G13} can be bounded using the property \hyperlink{(ii)}{\rm(ii)}. Consequently,  choosing $\delta$ sufficiently small, plugging the estimates \eqref{G14} and \eqref{G17} into the equality \eqref{G13}, integrating  the resulting inequality from $1$ to $t+1$, and using the Cauchy--Schwarz inequality, we~obtain
        \begin{align*}
            \notag&\|\psi(t)u(t+1)\|^2+\int_0^t\left(\nu\|\psi(s)\nabla u(s+1)\|^2+a\|\psi(s)u(s+1)\|^2\right)\mathrm{d}s\\&\qquad\qquad\qquad\le C_{a,\nu}\int_1^{t+1}\left(\|u\|^2\|\nabla u\|^2+\|\nabla u\|^2+\|u\|^6\right)\mathrm{d}s\notag\\&\qquad\qquad\qquad\quad +C_{a,\nu,h,\mathcal{B}_{\varphi}}t+M_{\tilde\psi}(t).
        \end{align*}
        As the quadratic variation $\langle M_{\tilde\psi}\rangle$ satisfies
        \begin{align*}\langle M_{\tilde\psi}\rangle(t)&=4\int_1^{t+1}\sum_{j=1}^{\infty}b_j^2\langle \tilde\psi u,\tilde\psi e_j\rangle^2\mathrm{d}s\le 4\mathcal{B}_{\varphi}\int_1^{t+1}\|\tilde{\psi}u\|^2\mathrm{d}s\\&= 4\mathcal{B}_{\varphi}\int_0^{t}\|\psi(s)u(s+1)\|^2\mathrm{d}s,
        \end{align*}
          by setting $\tilde\gamma':=\frac{a}{4\mathcal{B}_{\varphi}}$, we get
        \begin{align*}
            \tilde{\mathcal{E}}_{\psi}^u(t)&\le C_{a,\nu}\left(\mathcal{E}^u_1(t+1)+\mathcal{E}^u_2(t+1)\notag+\mathcal{E}^u_3(t+1)\right)\\&\quad+C_{a,\nu,h,\mathcal{B}_{\varphi}}t+C_{a,\nu}\left(M_{\tilde\psi}(t)-\frac{\tilde\gamma'}{2}\langle M_{\tilde\psi}\rangle(t)\right).
        \end{align*}
        For $p\ge 1$, let $(\kappa_p,\mathcal{C}_p)$ be the constants in Proposition \ref{proposition1}. Define 
        \[\notag\tilde\kappa:=C_{a,\nu,h,\mathcal{B}_{\varphi}}+C_{a,\nu}
        \left(\kappa_1+\kappa_2+\kappa_3\right),\]
        and let $\tilde{\mathcal{C}}$ be such that 
        \[C_{a,\nu}\left(\kappa_1+\kappa_2+\kappa_3+\mathcal{C}_{1}\|u_0\|^2+\mathcal{C}_{2}\|u_0\|^4+\mathcal{C}_{3}\|u_0\|^6\right)\le \tilde{\mathcal{C}}\left(1+\|u_0\|^6\right).\]
        Then,
        \begin{align*}
            \tilde{\mathcal{E}}_{\psi}^u(t)-\tilde\kappa t-\tilde{\mathcal{C}}\left(1+\|u_0\|^6\right)&\le  C_{a,\nu}\left(\mathcal{E}^u_{1}(t+1)-\kappa_{1}(t+1)-\mathcal{C}_{1}\|u_0\|^2\right)\\&\quad+ C_{a,\nu}\left(\mathcal{E}^u_{2}(t+1)-\kappa_{2}(t+1)-\mathcal{C}_{2}\|u_0\|^4\right)\\&\quad+ C_{a,\nu}\left(\mathcal{E}^u_{3}(t+1)-\kappa_{3}(t+1)-\mathcal{C}_{3}\|u_0\|^6\right)\\&\quad+C_{a,\nu}\left(M_{\tilde\psi}(t)-\frac{\tilde\gamma'}{2}\langle M_{\tilde\psi}\rangle(t)\right).
        \end{align*}
        Applying the exponential supermartingale estimate (cf. (A.57) in \cite{KS12}) and Proposition \ref{proposition1}, we conclude that
        \begin{align*}
            \mathbb{P}&\left\{\sup_{t\ge 0}\left(\tilde{\mathcal{E}}^u_{\psi}(t)-\tilde{\kappa}t-\tilde{\mathcal{C}}\left(1+\|u_0\|^6\right)\right)\ge \rho\right\}\\&\qquad\qquad\le\mathbb{P}\left\{\sup_{t\ge 0}\left(\mathcal{E}^u_{1}(t)-\kappa_{1}t-\mathcal{C}_{1}\|u_0\|^2\right)\ge \frac{\rho}{4C_{a,\nu}}\right\}\\&\qquad\qquad\quad+\mathbb{P}\left\{\sup_{t\ge 0}\left(\mathcal{E}^u_{2}(t)-\kappa_{2}t-\mathcal{C}_{2}\|u_0\|^4\right)\ge \frac{\rho}{4C_{a,\nu}}\right\}\\&\qquad\qquad\quad+\mathbb{P}\left\{\sup_{t\ge 0}\left(\mathcal{E}^u_{3}(t)-\kappa_{3}t-\mathcal{C}_{3}\|u_0\|^6\right)\ge \frac{\rho}{4C_{a,\nu}}\right\}\\&\qquad\qquad\quad+\mathbb{P}\left\{\sup_{t\ge 0}\left(M_{\tilde\psi}(t)-\frac{\tilde\gamma'}{2}\langle M_{\tilde\psi}\rangle(t)\right)\ge \frac{\rho}{4C_{a,\nu}}\right\}\\&\qquad\qquad\le C_{a,\nu,h,\mathcal{B}_\varphi,q}\left(\mathbb{E}\|u_0\|^{5q}+1\right)\left(e^{-\tilde\gamma \rho}+\frac{1}{\rho^{\frac{q}{2}-1}}\right),
        \end{align*}
        where $\tilde\gamma:=\frac{\tilde\gamma'}{4C_{a,\nu}}$.   
           \end{proof}
        
        \subsection{Weighted estimate for the vorticity}\label{GEWVorticity}
        Our next goal is to derive a growth estimate for the space-time weighted energy functional $\tilde\EE^u_{1,\psi}$ defined in \eqref{GC2-3}.
        \begin{proposition}\label{proposition3}There exist constants $\tilde{\kappa}_1,\tilde{\gamma}_1,\tilde{\mathcal{C}}_1$ depending on $a,\nu,h,\mathcal{B}_{\varphi}$ such that for any $q,\rho>2$,
            \begin{align*}
            &\mathbb{P}\left\{\sup_{t\ge 0}\left(\tilde{\mathcal{E}}^u_{1,\psi}(t)-\tilde{\kappa}_1t-\tilde{\mathcal{C}}_1(1+\|u_0\|^6+\|w(1)\|^4)\right)\ge \rho\right\}\\&\qquad\qquad\qquad\qquad\qquad\qquad\le  C_{a,\nu,h,\mathcal{B}_\varphi,q}\left(\mathbb{E}\|u_0\|^{8(q+1)}+1\right)\left(e^{-\tilde{\gamma}_1\rho}+\frac{1}{\rho^{\frac{q}{2}-1}}\right).
        \end{align*}
        \end{proposition}
         \begin{proof} Again, we set $\tilde{\psi}(t+1):=\psi(t)$ and proceed as in the proof of Proposition~\ref{proposition2}. By the It\^o formula and the equation~\eqref{G11_0},  
        \begin{align}
            \label{G31}
            \mathrm{d}\|\tilde\psi w\|^2&=2\langle\partial_t\tilde\psi w,\tilde\psi w\rangle\dd t+2\langle \tilde\psi w, \tilde\psi(-u\cdot \nabla w+ \nu\Delta w-aw+\curl h)\rangle\mathrm{d}t\notag\\&\quad+\mathrm{d}M_{1,\tilde\psi}(t)+\sum_{j=1}^{\infty}b_j^2\|\tilde\psi \curl e_j\|^2\mathrm{d}t,
        \end{align}
        where
        \[M_{1,\tilde\psi}(t):=2\int_1^{t+1}\sum_{j=1}^{\infty}b_j\langle \tilde\psi w,\tilde\psi \curl e_j\rangle\mathrm{d}\beta_j.\]
        For the nonlinear convection term, we use the property \hyperlink{(ii)}{\rm(ii)}  of the weight function $\psi$ and the Ladyzhenskaya inequality \eqref{lady} to derive 
        \begin{align}
            \label{G32}
            -\langle\tilde\psi w,\tilde\psi  u \cdot\nabla w\rangle&= \langle\nabla \tilde\psi \cdot u ,\tilde\psi w^2\rangle\notag\\&\le C\| \tilde\psi w \|_{L^4}\|u\|_{L^4}\|
            w\|\notag\\&\le C\| \tilde\psi w \|^{\frac{1}{2}}\| \nabla(\tilde\psi w )\|^{\frac{1}{2}}\|u\|^{\frac{1}{2}}\|w\|^{\frac{3}{2}}\notag\\&\le \delta\| \tilde\psi w \|^{2}+\delta\| \nabla(\tilde\psi w )\|^{2}+C\|w\|^{4}+C\|u\|^{2}\|\nabla
            u\|^{2}\notag\\&\le \delta\| \tilde\psi w \|^{2}+\delta\| \tilde\psi \nabla w \|^{2}+C\left(1+\|w\|^{4}+\|u\|^{2}\|\nabla
            u\|^{2}\right),
        \end{align}
        where $\delta>0$ is a small parameter to be determined. For the deterministic forcing term,  integrating by parts, we obtain
        \begin{align}
            \label{G32-1}\langle\tilde\psi w,\tilde\psi\curl h\rangle&=\langle\nabla^{\perp}(\tilde\psi^2 w),h\rangle\le \|\tilde\psi\nabla w\|\,\|\tilde\psi h\|+C\|w\|\,\|\tilde\psi h\|\notag\\&\le \delta\|\tilde\psi\nabla w\|^2+C\left(\|\tilde\psi h\|^2+\|w\|^2\right),
        \end{align}
        where we used the property (ii) of the weight function $\psi$. Moreover, by the property~\hyperlink{(i)}{\rm(i)} of the weight function $\psi$, we have 
        \begin{align}\label{G41}
            \langle M_{1,\tilde\psi}\rangle (t)=4\int^{t+1}_1\sum_{j=1}^{\infty}b_j^2\langle\tilde\psi w,\tilde\psi \curl e_j\rangle^2\mathrm{d}s\le 4\mathcal{B}_{\varphi}\int_1^{t+1}\|\tilde\psi w\|^2\mathrm{d}s.
        \end{align}
        Taking $\delta:=  (a\wedge\nu)/4$, plugging \eqref{G32}--\eqref{G41} into~\eqref{G31}, integrating the resulting inequality from $1$ to $t+1$, and using the Cauchy--Schwarz inequality, we derive 
        \begin{align*}
         \|\psi(t) w(t+1)\|^2& +\int_0^t\left({ \nu}\|\psi(s)\nabla w(s+1)\|^2+{ a}\|\psi(s) w(s+1)\|^2\right)\mathrm{d}s\\&\le C_{a,\nu}\int_1^{t+1}\left(\|w\|^4+\|u\|^{2}\|\nabla
            u\|^{2}\right)\mathrm{d}s+C_{a,\nu,h,\mathcal{B}_{\varphi}}t\\&\quad+C_{a,\nu}\left(M_{1,\tilde\psi}(t)-\frac{\tilde{\gamma}_1'}{2}\langle M_{1,\tilde\psi}\rangle(t)\right),
        \end{align*}
        where $\tilde{\gamma}_1':=\frac{a}{4\mathcal{B}_{\varphi}}$. This implies 
        \begin{align}\label{G42}
           \notag\tilde{\mathcal{E}}^u_{1,\psi}(t)&\le C_{a,\nu}\left(\mathcal{E}_{1,2}^{u}(t)+\mathcal{E}_2^{u}(t+1)\right)+C_{a,\nu,h,\mathcal{B}_{\varphi}}t\\&\qquad\qquad\qquad\qquad\qquad+C_{a,\nu}\left(M_{1,\tilde\psi}(t)-\frac{\tilde{\gamma}_1'}{2}\langle M_{1,\tilde\psi}\rangle(t)\right).
        \end{align}
        For $p\ge 1$, let $\kappa_p,\kappa_{1,p}$ be the constants in Propositions \ref{proposition1} and \ref{proposition1_1}, respectively, and let
        \[\tilde{\kappa}_1:=C_{a,\nu}\left(\kappa_{1,2}+\kappa_2\right)+C_{a,\nu,h,\mathcal{B}_{\varphi}}\]
        and choose $\tilde{\mathcal{C}}_1$ such that
        \[C_{a,\nu}\mathcal{C}_{1,2}\|w(1)\|^4+C_{a,\nu}\mathcal{C}_2\|u_0\|^4+C_{a,\nu}\kappa_2\le \tilde{\mathcal{C}}_1\left(1+\|u_0\|^6+\|w(1)\|^4\right).\]
        Then, we derive from \eqref{G42} that
        \begin{align*} 
            \notag\tilde{\mathcal{E}}^u_{1,\psi}&(t)-\tilde{\kappa}_1t-\tilde{\mathcal{C}}_1\left(1+\|u_0\|^6+\|w(1)\|^4\right)\\&\le C_{a,\nu}\left(M_{1,\tilde\psi}(t)-\frac{\tilde{\gamma}_1'}{2}\langle M_{1,\tilde\psi}\rangle(t)\right)+C_{a,\nu}\left(\mathcal{E}_{1,2}^{u}(t)-\kappa_{1,2}t-\mathcal{C}_{1,2}\|w(1)\|^4\right)\notag\\&\quad+C_{a,\nu}\left(\mathcal{E}_2^{u}(t+1)-\kappa_2(t+1)-\mathcal{C}_2\|u_0\|^4\right),
        \end{align*}
        which, by the exponential supermartingale estimate, implies 
        \begin{align*}
           \notag &\mathbb{P}\left\{\sup_{t\ge 0}\left(\tilde{\mathcal{E}}^u_{1,\psi}(t)-\tilde{\kappa}_1t-\tilde{\mathcal{C}}_1(1+\|u_0\|^6+\|w(1)\|^4)\right)\ge \rho\right\}\\&\qquad\qquad\qquad\le \mathbb{P}\left\{\sup_{t\ge 0}\left(\mathcal{E}_{1,2}^{u}(t)-\kappa_{1,2}t-\mathcal{C}_{1,2}\|w(1)\|^4\right)\ge \frac{\rho }{3C_{a,\nu}}\right\}\notag\\&\quad\qquad\qquad\qquad+ \mathbb{P}\left\{\sup_{t\ge 0}\left(\mathcal{E}_2^{u}(t)-\kappa_2t-\mathcal{C}_2\|u_0\|^4\right)\ge \frac{\rho }{3C_{a,\nu}}\right\}\notag\\&\qquad\quad\qquad\qquad + \mathbb{P}\left\{\sup_{t\ge 0}\left(M_{1,\tilde\psi}(t)-\frac{\tilde{\gamma}_1'}{2}\langle M_{1,\tilde\psi}\rangle(t)\right)\ge \frac{\rho }{3C_{a,\nu}}\right\}\notag\\&\qquad\qquad\qquad\le  C_{a,\nu,h,\mathcal{B}_\varphi,q}\left(\mathbb{E}\|u_0\|^{8(q+1)}+1\right)\left(e^{-\tilde{\gamma}_1\rho}+\frac{1}{\rho^{\frac{q}{2}-1}}\right),
        \end{align*}
        where $\tilde{\gamma}_1:=\frac{\tilde{\gamma}_1'}{3C_{a,\nu}}$.  
        \end{proof}

        \subsection{Estimate for the weighted parabolic energy}
       Finally, we combine the estimates established in Sections~\ref{GEV}--\ref{GEWVorticity} to derive a growth estimate for the weighted parabolic energy $\mathcal{E}^u_{\psi}(t)$ defined in \eqref{G1}. 
        \begin{proposition}\label{proposition4}
        There exist constants $\kappa,\gamma,\mathcal{C}$ depending on $a,\nu,h,\mathcal{B}_{\varphi}$ such that for any $q,\rho>2$,
            \begin{align*}
            &\mathbb{P}\left\{\sup_{t\ge 0}\left(\mathcal{E}^u_{\psi}(t)-\kappa t-\mathcal{C}\left(1+\|u_0\|^6+\|w(1)\|^4\right)\right)\ge \rho\right\}\\&\qquad\qquad\qquad\qquad\qquad\qquad\le C_{a,\nu,h,\mathcal{B}_\varphi,q}\left(\mathbb{E}\|u_0\|^{8(q+1)}+1\right)\left(e^{- \gamma\rho}+\frac{1}{\rho^{\frac{q}{2}-1}}\right).
        \end{align*}
        \end{proposition}
        \begin{proof}
            By definition,
            \[\mathcal{E}^u_{\psi}(t)\le\mathcal{E}^u_{1}(t)+\mathcal{E}^u_{3}(t)+\I_{t\ge 1}\left(\mathcal{E}^u_{1,1}(t-1)+\tilde{\mathcal{E}}^u_{\psi}(t-1)+\tilde{\mathcal{E}}_{1,\psi}^u(t-1)\right),\]
            which implies 
            \begin{align*}
                \mathcal{E}^u_{\psi}(t)-\left(\kappa_1+\kappa_3+\tilde{\kappa}+\kappa_{1,1}+\tilde{\kappa}_1\right)t&\le (\mathcal{E}^u_1-\kappa_1t)+(\mathcal{E}^u_3-\kappa_3t)\\&\quad+\I_{t\ge 1}(\mathcal{E}^u_{1,1}(t-1)-\kappa_{1,1}(t-1))\\&\quad+\I_{t\ge 1}(\tilde{\mathcal{E}}^u_{\psi}(t-1)-\tilde{\kappa}( t-1))\\&\quad+\I_{t\ge 1}(\tilde{\mathcal{E}}^u_{1,\psi}(t-1)-\tilde{\kappa}_1(t-1)).
            \end{align*}
            Consequently, there exists a constant $\mathcal{C}$ depending on $a,\nu,h, \mathcal{B}_{\varphi}$ such that 
            \begin{align*}
                \mathcal{E}^u_{\psi}(t)-&(\kappa_1+\kappa_3+\tilde{\kappa}+\kappa_{1,1}+\tilde{\kappa}_1)t-\mathcal{C}\left(1+\|u_0\|^6+\|w(1)\|^4\right)\\&\quad \le \left(\mathcal{E}^u_1-\kappa_1t-\mathcal{C}_1\|u_0\|^2\right)+\left(\mathcal{E}^u_3-\kappa_3t-\mathcal{C}_3\|u_0\|^6\right)\\&\qquad+\I_{t\ge 1}(\mathcal{E}^u_{1,1}(t-1)-\kappa_{1,1}(t-1)-\mathcal{C}_{1,1}\|w(1)\|^2)\\&\qquad+\I_{t\ge 1}(\tilde{\mathcal{E}}^u_{\psi}(t-1)-\tilde{\kappa}( t-1)-\tilde{\mathcal{C}}\left(1+\|u_0\|^6\right))\\&\qquad+\I_{t\ge 1}(\tilde{\mathcal{E}}^u_{1,\psi}(t-1)-\tilde{\kappa}_1(t-1)-\tilde{\mathcal{C}}_1\left(1+\|u_0\|^6+\|w(1)\|^4\right)).
            \end{align*}
            By setting 
            \[\kappa:=\kappa_1+\kappa_3+\tilde{\kappa}+\kappa_{1,1}+\tilde{\kappa}_1,\qquad \gamma:=\frac{\tilde\gamma\wedge\tilde{\gamma}_1}{5},\]
            and applying Propositions \ref{proposition1}, \ref{proposition1_1}, \ref{proposition2}, and \ref{proposition3}, we conclude the proof.
        \end{proof}
         As a corollary, we now establish an estimate for the distribution function of the stopping time 
        \begin{equation}\label{E:tau1}
    	         \tau^u_1:=\inf\left\{t\ge0:\mathcal{E}_{\psi}^u(t)\ge (K+2L)t+2\rho+\mathscr{C}\left(1+\|u_0\|^6\right)\right\},
        \end{equation}
        where $K,\mathscr{C}, \rho$, and $L$ are some parameters to be determined later.
        \begin{corollary}\label{corollary2}
        Let $\kappa,\gamma,\mathcal{C}$ be given in Proposition \ref{proposition4}. If $K\ge \kappa$ and $\mathscr{C}\ge \mathcal{C}$, then
        \[\mathbb{P}\{l\le \tau^u_1<\infty\}\le C_{a,\nu,h,\mathcal{B}_{\varphi},q}\left(\mathbb{E}\|u_0\|^{8(q+1) }+1\right)\left(e^{- \gamma(\rho+Ll)}+\frac{1}{(\rho+Ll)^{\frac{q}{2}-1}}\right)\]
        for any $q,\rho>2$ and $L,l\ge 0$.
        \end{corollary}
        \begin{proof} Note that 
        \begin{align*}
            \left\{l\le \tau^u_1<\infty\right\}&\subset \left(\left\{l\le \tau^u_1<\infty\right\}\mcap\left\{\|w(1)\|\le \mathcal{C}^{-\frac{1}{4}}(\rho+Ll)^{\frac{1}{4}}\right\}\right)\\&\quad\mcup \left\{\|w(1)\|>\mathcal{C}^{-\frac{1}{4}}(\rho+Ll)^{\frac{1}{4}}\right\}.
        \end{align*}
       By the definition of $\tau^u_1$,  
        \begin{align*}
            \mathcal{E}^u_{\psi}(\tau^u_1)&  \ge \left(K+2L\right)\tau^u_1+2\rho+\mathscr{C}\left(1+\|u_0\|^6\right)\\&\ge \kappa \tau^u_1+\mathcal{C}\left(1+\|u_0\|^6+\|w(1)\|^4\right)+\rho+Ll
        \end{align*}
        on the event 
        $$\left\{l\le \tau^u_1<\infty\right\}\mcap\left\{\|w(1)\|\le \mathcal{C}^{-\frac{1}{4}}(\rho+Ll)^{\frac{1}{4}}\right\}.
        $$ Therefore,
       \begin{align*}
            &\left\{l\le \tau^u_1<\infty\right\}\subset  \left\{\|w(1)\|>\mathcal{C}^{-\frac{1}{4}}(\rho+Ll)^{\frac{1}{4}}\right\}\\&\qquad\qquad\qquad
            \mcup\left\{\sup_{t\ge 0}\left(\mathcal{E}^u_{\psi}(t)-\kappa t-\mathcal{C}\left(1+\|u_0\|^6+\|w(1)\|^4\right)\right)\ge \rho+Ll\right\}.
        \end{align*}
        Applying Proposition \ref{proposition4} and Lemma \ref{lemmaA1}, we obtain the desired result.
        \end{proof}
        We introduce another stopping time $\tau_2^u$ in order to control the growth of~$\mathcal{E}^u_1(t)$:
    \begin{equation}\label{E:tau2}
            \tau_2^u:=\inf\left\{t\ge 0: \mathcal{E}^u_1(t)\ge (K+L)t+\rho+\mathscr{C}\|u_0\|^2\right\}.
    \end{equation}
        \begin{corollary}\label{corollary3}
            Let $\kappa_1,\mathcal{C}_1$ be the constants in Proposition \ref{proposition1}. If $K\ge \kappa_1$ and $\mathscr{C}\ge \mathcal{C}_1$, then
            \[\mathbb{P}\left\{l\le \tau^u_2<\infty\right\}\le C_{a,\nu,h,\mathcal{B}_\varphi,q}\left(\mathbb{E}\|u_0\|^q+1\right)\frac{1}{(\rho+Ll)^{\frac{q}{2}-1}}\]
            for any $q,\rho>2$ and $L,l\ge0$.
        \end{corollary}
        \begin{proof}
            Notice that 
            \[\mathcal{E}^u_1(\tau_2^u)   \ge (K+L)\tau_2^u+\rho+\mathscr{C}\|u_0\|^2\ge\kappa_1\tau_2^u+\mathcal{C}_1\|u_0\|^2+\rho+Ll \]
            on the event $\{l\le \tau^u_2<\infty\}$, which implies
            \[\left\{l\le \tau^u_2<\infty\right\}\subset \left\{\sup_{t\ge0}\left(\mathcal{E}^u_1(t)-\kappa_1t-\mathcal{C}_1\|u_0\|^2\right)\ge \rho+Ll\right\}.\]
            Hence, an application of Proposition \ref{proposition1} yields the result.
        \end{proof}
        Combining Corollaries \ref{corollary2} and \ref{corollary3}, we can simultaneously control the growth of both $\mathcal{E}^u_{\psi}(t)$ and $\mathcal{E}^u_{1}(t)$ via the stopping time
        \begin{align}
            \label{G47}
            \tau^u:=\tau_1^u\wedge\tau_2^u.
        \end{align}
        Note that 
        \[\left\{l\le \tau^u<\infty\right\}\subset \left\{l\le \tau_1^u<\infty\right\}\mcup \left\{l\le \tau_2^u<\infty\right\},\]
        which allows us to derive the following estimate for the distribution of $\tau^u$.
        \begin{corollary}\label{corollary1}
        Let $\kappa,\gamma,\mathcal{C}$ be given in Proposition~\ref{proposition4} and $\kappa_1,\mathcal{C}_1$  in Proposition~\ref{proposition1}. If $K\ge \kappa\vee\kappa_1$ and $\mathscr{C}\ge \mathcal{C}\vee \mathcal{C}_1$, then
        \[\mathbb{P}\{l\le \tau^u<\infty\}\le C_{a,\nu,h,\mathcal{B}_{\varphi},q}\left(\mathbb{E}\|u_0\|^{8(q+1) }+1\right)\left(e^{- \gamma(\rho+Ll)}+\frac{1}{(\rho+Ll)^{\frac{q}{2}-1}}\right)\]
        for any $q,\rho>2$ and $L,l\ge 0$.
        \end{corollary}

        \section{Stability of solutions}\label{STA}
        We begin this section by establishing a Foia\c{s}--Prodi type estimate for the NS system \eqref{INTRO1}. We then derive a growth estimate for an auxiliary process that arises in the course of this estimate.

           \subsection{The Foia\c{s}--Prodi type estimate}
       The Foia\c{s}--Prodi type estimate is stated in Proposition~\ref{proposition5}. One of the key ingredients in its proof is the following truncated Poincar\'e inequality.
    
        \begin{lemma}\label{lemma1}
            For any $\epsilon>0$ and $A>1$, there is an integer $N_0\ge1$ such that 
    	$$
            \|\mathrm{Q}_N\Pi(\chi_A f)\|\le \epsilon \|f\|_{H^{\frac{1}{2}}},
    $$
    		{ for any $f\in H^\frac12$ and $N\ge N_0$}, where $\chi_A:\mathbb{R}^2\to [0,1]$ is a smooth cut-off function satisfying 
            \[\chi_A(x)=\begin{cases}
                1,\qquad x\in B(0,A),\\
                0,\qquad x\notin B(0,2A)
            \end{cases}\]
            with $B(0,A):=\{x: |x|<A\}$ and $\mathrm{Q}_N:=\mathrm{I}-\mathrm{P}_N$ with $\mathrm{P}_N$ the same projector as in \eqref{MR7}. 
    	\end{lemma} 
        This lemma can be proved by combining the compactness of the operator 
        \[\LL: f\mapsto \Pi(\chi_Af),\qquad H^{\frac{1}{2}}\to H,\]
        and the fact that $\sup_{f\in K}\|\mathrm{Q}_Nf\|\to0$ for any compact subset $K$ of $H$; see Lemma~2.1 in \cite{NZ24} for details. We will also use the following weighted estimates, whose proof is deferred to the Appendix.
        \begin{lemma}
            \label{lemma6} 
            There is a constant $C>0$ such that 
            
            \begin{itemize}
             \item[\hypertarget{(a)}{\bf(a)}] for any $f\in L^2(\mathbb{R}^2;\R^2)$,  
            	 \[
            \|\psi (t) \Pi f\|\le C\|\psi (t) f\|, \quad t\ge 1;  
            \]
                    \item[\hypertarget{(b)}{\bf(b)}]       for any $u\in H^2$,                 	 \[\|
             \psi (t) \nabla^2u\|\le C\left(\|\nabla u\|+\|\psi (t) \nabla w\|\right), \quad t>0,
             \]where $w:=\curl u$.
            \end{itemize}

                \end{lemma}

             \begin{proposition}\label{proposition5} Let $g(t):=u(t)-v(t)$, where $\{u(t)\}_{t\ge 0}$ is the solution of \eqref{INTRO1} issued from $u\in H$ and $\{v(t)\}_{t\ge0}$ is the solution of \eqref{MR7}. Then, for any $\epsilon>0$, there is a time $T\ge2$ and an integer $N_0\ge1$ depending only on $\epsilon$ such that 
        \begin{align}\label{STA2}
            &\|g(t+T)\|^2\le \left(\|g(s+T)\|^2+C\|\mathrm{P}_Ng(s+T)\|^2_{H^2}\right)e^{-a(t-s)}\notag\\&\ \times\exp\left(C\epsilon\int_{s+T}^{t+T}\left(\|v(r)\|^2_{H^1}+\|\psi(r-1) \nabla v(r)\|^2+\|\psi(r-1) v(r)\|^2\right)\mathrm{d}r\right)\notag\\&\  \times \exp\left(C\epsilon\int_{s+T}^{t+T}\left(\|w(r)\|^2_{H^1}+\|\psi(r-1)\nabla u(r)\|^2+\|\psi(r-1) \nabla w(r)\|^2\right)\mathrm{d}r\right)
        \end{align}
    	for any $t\ge s\ge 0 $ and $N\ge N_0$, where $C>0$ is a constant depending on $a,\nu$.
    	\end{proposition}
        \begin{proof}Without loss of generality, we assume that $\epsilon\le 1$. 
            The equation for $g$ is given by 
            \begin{align}\label{STA2_1}
            \begin{cases}
            \partial_t g+ag+\mathrm{Q}_N[\Pi(u\cdot\nabla )u-\Pi(v\cdot\nabla )v-\nu\Delta g]=0,\\
                g|_{t=0}=u-u'.
            \end{cases}
            \end{align}
            Taking the inner product in $H$ of this equation 
            with $g$, we obtain
            \begin{align}\label{STA3}
            \frac{1}{2}\frac{\mathrm{d}}{\mathrm{d} t}\|g(t)\|^2+a\|g(t)\|^2&=-\langle g,\mathrm{Q}_N[\Pi(g\cdot\nabla )u+\Pi(v\cdot\nabla )g-\nu\Delta g]\rangle\notag\\&=:-I_1-I_2-I_3.
        \end{align}
        To estimate $I_1$, we decompose
        \begin{align*}
            I_1=\langle g, \mathrm{Q}_N\Pi(\chi_A\Pi(g\cdot\nabla )u)\rangle+\langle \mathrm{Q}_N g, (1-\chi_A)\Pi(g\cdot\nabla )u\rangle=:I_{11}+I_{12}.
        \end{align*}
        To bound $I_{11}$, we apply Lemma \ref{lemma1}, the boundedness of the Leray projector in~$H^{\frac{1}{2}}$, the Kato--Ponce inequality (e.g., see Proposition~1.1 in Chapter~2 of~\cite{T-2000})  
         \[\||\nabla|^{\frac{1}{2}}(f_1f_2) \|\lesssim \||\nabla|^{\frac{1}{2}}f_1\|_{L^4}\|f_2\|_{L^4}+\||\nabla|^{\frac{1}{2}}f_2\|_{L^4}\|f_1\|_{L^4},\]
        the embedding $H^{\frac{1}{2}}\hookrightarrow L^4$, and the estimate
        \[\|\nabla^j u\|\lesssim\|\nabla^{j-1}w\|,\qquad\forall j\ge 1\]
        to derive
        \begin{align*}
             I_{11}&\le C\epsilon\|g\|\|(g\cdot \nabla )u\|_{H^{\frac{1}{2}}}\notag\\&\le C\epsilon\|g\|\left(\|(g\cdot \nabla )u\|+\||\nabla|^{\frac{1}{2}}g\|_{L^4}\|\nabla u\|_{L^4}+\||\nabla|^{\frac{1}{2}}\nabla u\|_{L^4}\|g\|_{L^4}\right)\notag\\&\le C\epsilon\|g\|\|g\|_{H^1}\|\nabla u\|_{H^1}\le C\epsilon\|g\|\|g\|_{H^1}\|w\|_{H^1}.
        \end{align*}
        To bound $I_{12}$, we choose $A$ and $T$ sufficiently large such that 
        \[\frac{1}{|\psi(t-1,x)|}\le \epsilon\]
        for any $t\ge T$ and $x\notin B(0,A)$
        and use properties \hyperlink{(a)}{\rm(a)} and \hyperlink{(b)}{\rm(b)}  in Lemma \ref{lemma6}:
        \begin{align*}
        I_{12}&\le \epsilon \|g\|\| \tilde\psi\Pi(g\cdot\nabla )u\| \le C \epsilon \|g\|\| \tilde\psi(g\cdot\nabla )u\|\\&\le C\epsilon\|g\|\|g\|_{H^1}\|\tilde\psi \nabla u\|_{H^1}\\&\le C\epsilon\|g\|\|g\|_{H^1}\left(\|w\|_{H^1}+\|\tilde\psi \nabla u\|+\|\tilde\psi \nabla w\|\right),
        \end{align*}
        where we denote $\tilde{\psi}(t):=\psi(t-1)$. Therefore,
        \begin{align}
            \label{STA4}
            I_1\le C\epsilon\|g\|\|g\|_{H^1}\left(\|w\|_{H^1}+\|\tilde\psi \nabla u\|+\|\tilde\psi \nabla w\|\right).
        \end{align}
        As for $I_2$, we use the cancellation property of the convection term and then decompose as in the case of $I_1$:
        \begin{align*}
           I_2&=\langle g,\mathrm{Q}_N\Pi(v\cdot\nabla )\mathrm{P}_Ng\rangle\\&=\langle g,\mathrm{Q}_N\Pi(\chi_A\Pi(v\cdot\nabla )\mathrm{P}_Ng)\rangle+\langle \mathrm{Q}_Ng,(1-\chi_A)\Pi(v\cdot\nabla )\mathrm{P}_Ng\rangle=:I_{21}+I_{22}.
        \end{align*}
       The estimates for $I_{21}$ and $I_{22}$ are similar to those for $I_{11}$ and $I_{12}$:
        \[I_{21}\le C\epsilon\|g\|\|(v\cdot\nabla)\mathrm{P}_Ng\|_{H^{\frac{1}{2}}}\le C\epsilon\|g\|\|\nabla \mathrm{P}_Ng\|_{H^1}\|v\|_{H^1}\]
        and 
        \begin{align*}
            \notag I_{22}&\le \epsilon \|g\|\|\tilde\psi\Pi(v\cdot\nabla )\mathrm{P}_Ng\|\le C\epsilon \|g\|\|\tilde\psi(v\cdot\nabla )\mathrm{P}_Ng\|\\&\le C\epsilon \|g\|\|\nabla \mathrm{P}_N g\|_{H^1}\left(\|\tilde\psi v\|+\|\tilde\psi\nabla v\|+\|v\|_{H^1}\right).
        \end{align*}
        Hence, we get
        \begin{align}
            \label{STA11}
            I_2\le C\epsilon\|g\|\|\mathrm{P}_Ng \|_{H^2}\left(\|\tilde\psi v\|+\|\tilde\psi\nabla v\|+\|v\|_{H^1}\right).
        \end{align}
        Finally, for $I_{3}$, we have
        \begin{align}
            \label{STA12}
            \notag-I_3&=\nu\langle g,\mathrm{Q}_N\Delta g\rangle=-\nu \|\nabla g\|^2-\nu\langle g,\mathrm{P}_N \Delta g\rangle\\&=-\nu \|\nabla g\|^2-\nu\langle \Delta \mathrm{P}_Ng, g\rangle\le -\nu \|\nabla g\|^2+\nu\|\Delta \mathrm{P}_N g\|\|g\|.
        \end{align}
        Plugging the estimates \eqref{STA4}--\eqref{STA12} into \eqref{STA3}, we obtain
        \begin{align*}
            &\notag\frac{\mathrm{d}}{\mathrm{d}t}\|g(t)\|^2+a\|g\|^2+\nu\|\nabla g\|^2\lesssim_{a,\nu}\|\mathrm{P}_N g\|_{H^2}^2\\&\quad\quad+\epsilon\|g\|^2\left(\|w\|^2_{H^1}+\|\tilde\psi \nabla u\|^2+\|\tilde\psi \nabla w\|^2+\|\tilde\psi v\|^2+\|\tilde\psi\nabla v\|^2+\|v\|^2_{H^1}\right).
        \end{align*}
        An application of the Gronwall inequality, together with the fact that
    		\begin{equation}\label{STA14}
    					\mathrm{P}_N g(t)=e^{-a(t-s)}\mathrm{P}_Ng(s),
    		\end{equation}
        yields \eqref{STA2}.
        \end{proof}

        \subsection{Growth estimate for an auxiliary process}\label{S:4.2}
        
         To handle the integral terms appearing in the Foia\c{s}--Prodi type inequality \eqref{STA2}, we establish an estimate for the stopping time $\tau^v$ defined in \eqref{G47}, where $\{v(t)\}$ is the solution of \eqref{MR7}. In what follows, we shall use the notations 
         \[\BB^*_N:=\sup_{1\le j\le N}b_j^{-2},\qquad D_N:=1+\sum_{j=1}^N\|e_j\|_{H^2}^2,\] 
         and fix constants $K$ and $\mathscr C$, depending only on $a,\nu,h,\mathcal B_\varphi$, sufficiently large so that Corollary~\ref{corollary1} holds. The values of $\rho$ and $L$ will be specified later in Section~\ref{proofoftheorem2}.

        \begin{proposition}\label{proposition6}{There are constants $\gamma, C>0$ and an integer $N_0\ge1$ depending on $a,\nu,h,\mathcal{B}_{\varphi},L$ such that if 
        \begin{align}
            \label{GEAP0}b_j\neq 0,\qquad\forall j=1,\ldots,N
        \end{align}
        for some $N\ge N_0$, then
        \begin{align}
            \label{GEAP1}
            \mathbb{P}\{\tau^v<\infty\}\le C_{a,\nu,h,\mathcal{B}_{\varphi},q}&\left(1+d^{8(q+1)}\right)\left(e^{- \gamma\rho}+\frac{1}{\rho^{\frac{q}{2}-1}}\right)\notag\\&+\frac{1}{2}\left(\exp\left(C\BB^*_ND_N^2\fd^2e^{C\left(\rho+d^6\right)}\right)-1\right)^{\frac{1}{2}}
        \end{align}
        for any $q,\rho>2$, $d,\fd>0$, and $u, u'\in \overline{B}_H(0,d)$ satisfying $\|u-u'\|\le \fd$.
       }
        \end{proposition}
         \begin{proof}{\it Step 1.} Fix $d,\fd>0$ and $u,u'\in\overline{B}_H(0,d)$ such that $\|u-u'\|\le\fd$. We begin by reducing the proof to an estimate involving a truncated version of $v$. More precisely, we denote by $\{u(t)\}_{t\ge 0}$ the solution of the problem \eqref{INTRO1} starting from $u$, and consider the truncated process $\{\hat{u}(t)\}_{t\ge0}$ defined as follows: for~$t\le \tau^u$, it coincides with $\{u(t)\}_{t\ge0}$; and for $t\ge \tau^u$, it solves the linear equation
        \begin{align}
            \label{GEAP2}
            \partial_t z+az=\nu\Delta z.
        \end{align}   
        We define the truncated processes $\{\hat{u}'(t)\}_{t\ge0}$ and $\{\hat{v}(t)\}_{t\ge0}$ in the same way with the stopping time $\tau^u$ replaced by $\tau^{u'}$ and $\tau^v$, respectively. Then, we have 
    	\[\{\tau^v<\infty\}=\{\tau^{\hat v}<\infty\}.\]
        Thus, it remains to control the probability~$\mathbb{P}\{\tau^{\hat{v}}<\infty\}$.

        \noindent{\it Step 2.} We now reduce the bound on $\mathbb{P}\{\tau^{\hat{v}}<\infty\}$ to an estimate involving a certain measurable transformation. Without loss of generality, we assume that the probability space $(\Omega,\mathscr{F},\mathbb{P})$ has the following structure: $\Omega:= C_0([0,\infty);H)$ is the space of continuous functions $f:[0,\infty)\to H$  vanishing at $t=0$, endowed with the topology of uniform convergence on compact sets, $\mathbb{P}$ is the distribution of the Wiener process
        \begin{equation}\label{GEAP4}
    	W(t):=\sum_{j=1}^{\infty}b_j\beta_j(t) e_j,
    	\end{equation}
    	  and
        $\mathscr{F}$ is the completion of the Borel $\sigma$-algebra of $\Omega$. 
        Now, let $N\ge1$ be an integer (with a lower bound to be specified in Step 4) and define a transformation $ \Phi^{u,u'}:  \Omega\to\Omega$ as follows:
    	\begin{equation}\label{GEAP5-1}
         \omega_t\mapsto\omega_t-\int_0^t\mathbb{I}_{\{s\le \tau^v\}}\mathrm{P}_N[\Pi(u\cdot\nabla )u-\Pi(v\cdot\nabla )v-\nu\Delta(u-v)]\mathrm{d}s.
        \end{equation}  
        Define $\mathfrak z(t,\omega):=\hat{u}'(t, \Phi^{u,u'}(\omega))$. It follows that $\mathfrak z(t,\omega)$ satisfies \eqref{MR7} up to the stopping time $\tau^{\hat v}$. Moreover, since the energy functionals in \eqref{G1} are non-anticipating, one has $\tau^{\mathfrak{z}}=\tau^{\hat v}$. Together with the pathwise uniqueness for the systems \eqref{INTRO1} and \eqref{GEAP2}, this implies 
    	\begin{align}\label{GEAP6}
    	\mathbb{P}\left\{\hat{u}'(t, \Phi^{u,u'}(\omega))=\hat{v}(t,\omega)\,\text{for any $t\ge0$}\right\}=1.
        \end{align}
        Thus,
        \begin{align}\label{GEAP7}
        \mathbb{P}\{\tau^{\hat{v}}<\infty\}=\Phi^{u,u'}_{*}\mathbb{P}\{\tau^{\hat{u}'}<\infty\}\le \mathbb{P}\{\tau^{\hat{u}'}<\infty\}+\|\mathbb{P}-\Phi^{u,u'}_*\mathbb{P}\|_{\textup {var}}\notag\\=\mathbb{P}\{\tau^{u'}<\infty\}+\|\mathbb{P}-\Phi^{u,u'}_*\mathbb{P}\|_{\textup {var}},\end{align}
    	where $\Phi^{u,u'}_*\mathbb{P}$ is the push-forward of $\mathbb{P}$ under $\Phi^{u,u'}$. The probability $\mathbb{P}\{\tau^{u'}<\infty\}$ is estimated in Corollary~\ref{corollary1}. 	
     
        \noindent {\it Step 3.} To bound the term $\|\mathbb{P}-\Phi^{u,u'}_*\mathbb{P}\|_{\textup {var}}$ in \eqref{GEAP7}, we employ the strategy in Section~3.3.3 of~\cite{KS12} based on the Girsanov theorem. Define an auxiliary transformation $\tilde \Phi^{u,u'}:\Omega\to\Omega$ by
    	\begin{equation}
         \omega_t\mapsto\omega_t-\int_0^t\mathbb{I}_{\{s\le \tau^v\wedge\tau^u\}}\mathrm{P}_N[\Pi(u\cdot\nabla )u-\Pi(v\cdot\nabla )v-\nu\Delta(u-v)]\mathrm{d}s.\label{GEAP5}
        \end{equation}  
        By the definition of the total variation metric, it follows that
        \begin{align}\label{GEAP7-1}
            \notag\|\mathbb{P}-\Phi^{u,u'}_*\mathbb{P}\|_{\textup {var}}&\le \|\mathbb{P}-\tilde\Phi^{u,u'}_*\mathbb{P}\|_{\textup {var}}+\|\tilde\Phi^{u,u'}_*\mathbb{P}-\Phi^{u,u'}_*\mathbb{P}\|_{\textup {var}}\\&\le \|\mathbb{P}-\tilde\Phi^{u,u'}_*\mathbb{P}\|_{\textup {var}}+\mathbb{P}\{\tilde\Phi^{u,u'}\neq \Phi^{u,u'}\}\notag\\&\le \|\mathbb{P}-\tilde\Phi^{u,u'}_*\mathbb{P}\|_{\textup {var}}+\mathbb{P}\{\tau^u<\infty\}.
        \end{align}
        Therefore, it remains to estimate $\|\mathbb{P}-\tilde\Phi^{u,u'}_*\mathbb{P}\|_{\textup {var}}$. To this end, we introduce the following decomposition of $\Omega$:
       \[\Omega=C_0([0,\infty);\mathrm{P}_NH)\oplus C_0([0,\infty);\mathrm{Q}_NH),\]
    	so that any $\omega\in\Omega$ can be expressed as $\omega=(\omega^{(1)},\omega^{(2)})$. Accordingly, the transformation~$\tilde\Phi^{u,u'}$ in \eqref{GEAP5} can be represented in the form  
    	\[\tilde\Phi^{u,u'}(\omega^{(1)},\omega^{(2)})=(\Psi^{u,u'}(\omega^{(1)},\omega^{(2)}),\omega^{(2)}),\]
    with $\Psi^{u,u'}: \Omega\to C_0([0,\infty);\mathrm{P}_NH)$ given by
    	$$
    		\Psi^{u,u'}(\omega^{(1)},\omega^{(2)})_t:=\omega_t^{(1)}+\int_0^t\mathcal{A}(s;\omega^{(1)},\omega^{(2)})\mathrm{d}s, 
    	$$
    	where  
    		\begin{equation}	
    		\label{GEAP8}
        \mathcal{A}(t):=-\mathbb{I}_{\{t\le \tau^v\wedge\tau^u\}}\mathrm{P}_N[\Pi(u\cdot\nabla )u-\Pi(v\cdot\nabla )v-\nu\Delta(u-v)].
    	\end{equation}
        Let $\mathbb{P}_N:=(\mathrm{P}_N)_*\mathbb{P}$ and $\mathbb{P}_N^{\perp}:=(\mathrm{Q}_N)_*\mathbb{P}$, where $\mathrm{P}_N$ and $\mathrm{Q}_N$ now denote the projections
        	\begin{gather*}
        		 \mathrm{P}_N:\Omega\to C_0([0,\infty);\mathrm{P}_NH),\\ \mathrm{Q}_N:\Omega\to C_0([0,\infty);\mathrm{Q}_NH).
        	\end{gather*}
        Lemma 3.3.13 in \cite{KS12} implies that 
    	\begin{align}\label{GEAP9}
        \|\mathbb{P}-\tilde\Phi^{u,u'}_*\mathbb{P}\|_{\textup{var}}\le \int_{C_0([0,\infty);\mathrm{Q}_NH)}\|\Psi^{u,u'}_*(\mathbb{P}_N,\omega^{(2)})-\mathbb{P}_N\|_{\textup{var}}\mathbb{P}^{\perp}_N(\mathrm{d}\omega^{(2)}).
        \end{align}
          By the Girsanov theorem, for each $\omega^{(2)}$, we have
        \begin{align}\label{GEAP10}
    	\|&\Psi^{u,u'}_*(\mathbb{P}_N,\omega^{(2)})-\mathbb{P}_N\|_{\textup{var}} \notag\\
    		&\qquad\qquad\le\frac{1}{2}\left(\left(\mathbb{E}_N\exp\left(6\BB_N^*\int_0^{\infty}\|\mathcal{A}(t;\cdot,\omega^{(2)})\|^2\mathrm{d} t\right)\right)^{\frac{1}{2}}-1\right)^{\frac{1}{2}},\end{align}
        provided that the  Novikov condition  
    		\begin{align}\label{GEAP11}
            \mathbb{E}_N\exp\left(c\int_0^{\infty}\|\mathcal{A}(t;\cdot,\omega^{(2)})\|^2\mathrm{d}t\right)<\infty
    		\end{align}is satisfied
    		for any $ c>0$ and $\omega^{(2)}$, where  $\mathbb{E}_N$ is the expectation with respect to $\mathbb{P}_N$. 
    
        \noindent{\it Step 4.} To verify that the Novikov condition holds, we need to establish a pathwise estimate for $\|g(t)\|^2$ up to the stopping time $\tau^u\wedge\tau^v$, where 
    	\[g(t):=u(t)-v(t)=\hat{u}(t)-\hat{v}(t), \quad t\in [0,\tau^u\wedge\tau^v).\]  
    	  Let $C_1$ be the constant in \eqref{STA2}. By Proposition \ref{proposition5} applied for 
        \begin{equation}\label{GEAP12}
    	\epsilon:=\frac{a}{4C_1(K+2L)},
    	\end{equation}  
    	there are $T\ge 2$ and $N_0\ge 1$ depending on $a,\nu,h,\mathcal B_\varphi,L$ such that \eqref{STA2} holds for any $N\ge N_0$.  We consider two cases.
    	
    	\noindent{\it Case 1: $\tau^u\wedge\tau^v\le T$.} Using the equality \eqref{STA3}, the estimate \eqref{STA12}, the nonlinear estimates
        \begin{align*}
            \notag\langle \mathrm{Q}_N g,\Pi(g\cdot\nabla )u\rangle&=\langle \mathrm{Q}_N g,(g\cdot\nabla )u\rangle\\&\le \|\nabla u\|\|g\|_{L^4}\|\mathrm{Q}_N g\|_{L^4}\notag\\&\le \|\nabla u\|\|g\|_{L^4}^2+ \|\nabla u\|\|g\|_{L^4}\|\mathrm{P}_N g\|_{L^4}\notag\\&\le C \|\nabla u\|\|g\|\|\nabla g\|+C\|\nabla u\|\|g\|\|\nabla\mathrm{P}_N g\|^{\frac{1}{2}}\|\nabla g\|^{\frac{1}{2}}\notag\\&\le\frac{\nu}{8}\|\nabla g\|^2+C_{\nu} \|g\|^2\|\nabla u\|^2+C_{\nu}\|\nabla\mathrm{P}_N g\|^{2}
        \end{align*}
        and 
        \begin{align*}
            \notag\langle \mathrm{Q}_N g,\Pi(v\cdot\nabla )g\rangle&=\langle \mathrm{Q}_N g,(v\cdot\nabla )\mathrm{P}_N g\rangle\\&\le C\|g\|\|v\|_{H^1}\|\mathrm{P}_Ng\|_{H^2}\le C\|g\|^2\|v\|_{H^1}^2+C\|\mathrm{P}_Ng\|_{H^2}^2,
        \end{align*}
        the Gronwall inequality, and \eqref{STA14}, we get 
        \begin{align}
            \label{GEAP14}
            \|g(t)\|^2&\le \left(\|u-u'\|^2+C_{a,\nu}\|\mathrm{P}_N (u-u')\|^2_{H^2}\right)\notag\\&\quad \times\exp\left(-at+C_{a,\nu}\int_0^t\left(\|u\|^2_{H^1}+\|v\|_{H^1}^2\right)\mathrm{d}s\right).
        \end{align}
        By the definition of $\tau^u\wedge\tau^v$,
        \begin{align}
        \label{GEAP15}
        \mathcal{E}_{\psi}^{u}(t)<  (K+2L)t+2\rho+\mathscr{C}(1+\|u\|^6)\le (K+2L)t+2\rho+\mathscr{C}(1+d^6),\\  
    	\label{GEAP16}
        \mathcal{E}_{\psi}^{v}(t)< (K+2L)t+2\rho+\mathscr{C}(1+\|u'\|^6)\le (K+2L)t+2\rho+\mathscr{C}(1+d^6)
    	\end{align}
        for $t< \tau^u\wedge\tau^v$. Hence,
        \begin{align}
            \label{GEAP17}
            \|g(t)\|^2&\le \left(\|u-u'\|^2+C_{a,\nu}\|\mathrm{P}_N (u-u')\|^2_{H^2}\right)\notag\\&\quad\times\exp\left(-at+C_{a,\nu}[(K+2L)T+2\rho+\mathscr{C}(1+d^6)]\right).
        \end{align}
        As 
        \begin{align*}\|\mathrm{P}_N (u-u')\|^2_{H^2} \le D_N\sum_{j=1}^N|\langle u-u',e_j\rangle|^2\le D_N\|u-u'\|^2, 
    	\end{align*}
       where $D_N:=1+\sum_{j=1}^N\|e_j\|_{H^2}^2,$ it follows from \eqref{GEAP17} that 
       \begin{align}
            \label{GEAP18}\|g(t)\|^2&\le CD_N\|u-u'\|^2 \exp(-at+C(\rho+d^6))\notag\\&\le CD_N\fd^2 \exp(-at+C(\rho+d^6)).
        \end{align}
        Here and below, $C$ is a constant depending on $a,\nu,h,\mathcal{B}_{\varphi},L$. 
    
        \noindent{\it Case 2: $\tau^u\wedge\tau^v> T$.} Inequality \eqref{GEAP18} holds true for $t\in[0,T]$. To estimate $\|g(t)\|^2$ over $[T,\tau^u\wedge\tau^v)$, we apply Proposition~\ref{proposition5} with the same $\epsilon$ as in~\eqref{GEAP12} and use \eqref{STA14}, \eqref{GEAP15}, and \eqref{GEAP16}:
        \begin{align*}
    		\|g(t)\|^2&\le \left(\|g(T)\|^2+C_1\|\mathrm{P}_N g(T)\|^2_{H^2}\right)\exp\left(-a(t-T)+C_1\epsilon(\mathcal{E}^u_{\psi}(t)+\mathcal{E}^v_{\psi}(t))\right)\\&\le \left(\|g(T)\|^2+C_1\|\mathrm{P}_N g(T)\|^2_{H^2}\right)\exp\left(-\frac{a}{2}t+aT+\frac{a(\rho+\mathscr{C}(1+d^6))}{(K+2L)}\right)
    	\end{align*}
        for $t\in[T,\tau^u\wedge\tau^v)$. Therefore, we obtain the bound 
        \begin{align}
            \label{GEAP19}
            \|g(t)\|^2&\le CD_N\|u-u'\|^2\exp\left(-\frac{a}{2}t+C(\rho+d^6)\right)\notag\\&\le  CD_N\fd^2\exp\left(-\frac{a}{2}t+C(\rho+d^6)\right), \quad t\in[T,\tau^u\wedge\tau^v).
        \end{align}

        \noindent{\it Step 5.} To verify the Novikov condition \eqref{GEAP11}, we first bound the terms on the right-hand side of \eqref{GEAP8}. Integrating by parts, we get
        \begin{align}\label{GEAP20}\|\mathrm{P}_N\Delta (u-v)\|^2=\sum_{j=1}^N\langle \Delta g,e_j\rangle^2=\sum_{j=1}^N\langle g,\Delta e_j\rangle^2\le CD_N\|g\|^2.
        \end{align}
        To estimate the term $\mathrm{P}_N[\Pi(u\cdot\nabla )u-\Pi(v\cdot\nabla )v]$, we use the identity   
           \[(u\cdot\nabla) u=\diver (u\otimes u),\]
        where $\otimes $ denotes the tensor product of vector fields:
        \begin{align}
            \label{GEAP21}
            \|\mathrm{P}_N[\Pi(u\cdot\nabla )u-\Pi(v\cdot\nabla )&v]\|^2= \sum_{j=1}^N\langle \diver \left(u\otimes u-v\otimes v\right) ,e_j\rangle^2\notag\\&\lesssim \sum_{j=1}^N\left(\langle{ g}\otimes {u},\nabla e_j\rangle^2+\langle {v}\otimes {g},\nabla e_j\rangle^2\right)\notag\\&\notag \lesssim\sum_{j=1}^N\left(\|g\|^2\|u\|_{H^1}^2\|\nabla e_j\|^2_{H^1}+\|g\|^2\|v\|_{H^1}^2\|\nabla e_j\|^2_{H^1}\right)\\&\le CD_N\|g\|^2\left(\|u\|^2_{H^1}+\|v\|_{H^1}^2\right).
        \end{align}
        Combining \eqref{GEAP20} and \eqref{GEAP21}, we derive \begin{align}
            \label{GEAP22}
            \|\mathcal{A}(t)\|^2\le CD_{N}\mathbb{I}_{\{t\le \tau^u\wedge\tau^v\}}\|g\|^2\left(1+\|u\|^2_{H^1}+\|v\|^2_{H^1}\right).
        \end{align}
        To further estimate the terms on the right-hand side of this inequality, we use \eqref{GEAP18} and \eqref{GEAP19}:
        \begin{align}\label{GEAP23}
            \int_0^{\tau^u\wedge\tau^v}\|g(t)\|^2\mathrm{d} t\le CD_N\fd^2e^{C\left(\rho+d^6\right)}.
            \end{align}
        Furthermore, integrating by parts, one has
        \begin{align}
            \label{GEAP24}
        \int_0^{\tau^u\wedge\tau^v}&\|g(t)\|^2\left(\|u\|^2_{H^1}+\|v\|^2_{H^1}\right)\mathrm{d} t\notag\\&\quad\le CD_N\fd^2e^{C(\rho+d^6)}\int_0^{\tau^u\wedge\tau^v}e^{-\frac{at}{2}}\mathrm{d} \int_0^t\left(\|u\|^2_{H^1}+\|v\|^2_{H^1}\right)\mathrm{d}s\notag\\&\quad\le  CD_N\fd^2e^{C(\rho+d^6)}e^{-\frac{a(\tau^u\wedge\tau^v)}{2}}\big(\mathcal{E}_{\psi}^u(\tau^u\wedge\tau^v)+\mathcal{E}_{\psi}^v(\tau^u\wedge\tau^v)\big)\notag\\&\quad\qquad+CD_N\fd^2e^{C(\rho+d^6)}\frac{a}{2}\int_0^{\tau^u\wedge\tau^v}e^{-\frac{at}{2}}\left(\mathcal{E}_{\psi}^u(t)+\mathcal{E}_{\psi}^v(t)\right)\mathrm{d}t\notag\\&\quad\le  CD_N\fd^2e^{C(\rho+d^6)}.
        \end{align}
        Finally, combining \eqref{GEAP22}--\eqref{GEAP24}, one gets the desired Novikov condition:
        \begin{align}\label{GEAP25}\mathbb{E}_N\exp\left(c\int_0^{\infty}\|\mathcal{A}(t;\cdot,\omega^{(2)})\|^2\mathrm{d} t\right)\le  \exp\left(cCD_N^2 \fd^2e^{C\left(\rho+d^6\right)}\right).
        \end{align}
         Combining the estimates \eqref{GEAP9}, \eqref{GEAP10}, and~\eqref{GEAP25}, we get  
        \begin{align}\label{GEAP25-1}
            \|\mathbb{P}-\tilde\Phi^{u,u'}_*\mathbb{P}\|_{\textup {var}}\le \frac{1}{2}\left(\exp\left(C\BB^*_ND_N^2\fd^2e^{C\left(\rho+d^6\right)}\right)-1\right)^{\frac{1}{2}}.
        \end{align}
        This, together with	\eqref{GEAP7}, \eqref{GEAP7-1}, and Corollary \ref{corollary1}, implies \eqref{GEAP1}.
        \end{proof}

        \section{Proof of Theorem \ref{theorem2}}\label{proofoftheorem2}
    
        Here we establish Theorem \ref{theorem2}, which, as previously mentioned,  implies Theorem~\ref{theorem1} and, consequently, the Main Theorem.
        \subsection{Recurrence}\label{recurrence}
        Before verifying the recurrence property, we first demonstrate that the Markov family $(u(t),\mathbb{P}_u)$ associated with \eqref{INTRO1} is irreducible in the sense that the trajectories enter any neighborhood of the steady state $\fu$ constructed below with positive probability. Recall that $P_t(u,\Gamma):=\mathbb{P}\{S_t(u,\cdot)\in \Gamma\}$ is the transition function of $(u(t),\mathbb{P}_u)$.
        \begin{lemma}\label{lemma2} 
        { There is an integer $N_0\ge 1$ depending on $a,\nu,h$ such that if \eqref{GEAP0} is satisfied for some $N\ge N_0$, then there exists a function $\fu\in H^1$ with $\|\fu\|\le 1$ depending on $a,\nu,h,N$, such that the following assertion holds. For any $R,d>0$, there is $T_0>0$ such that for every $T\ge T_0$,
        \begin{align}\label{R1}
        \inf_{u_0\in B_H(0,R)} P_{T}(u_0,B_{H}(\fu,d))\ge p>0.
        \end{align}
        Here, $T_0$ depends on $R,d,a$ and $p$ depends on $d,a,\nu,h,N,\{b_j\}_{j\ge 1},T$.
        }
    	\end{lemma}
        \begin{proof}{\textit{Step 1.} First, we construct the steady state $\fu$ in \eqref{R1}. Let $N\ge 1$ be an integer to be determined later; see \eqref{R8-1}. Define $\fu$ as the solution of the steady Navier--Stokes system
        \[\Pi(\fu\cdot\nabla)\fu-\nu\Delta\fu+a\fu=\fr_N,\qquad \fr_N:= \QQQ_N h.\]
        For all sufficiently large $N$, the above system can be solved through standard fixed point method. Moreover, by taking the inner product in $H$ on both sides, one has
        \begin{align*}
            \|\fu\|_{H^1}\lesssim_{a,\nu}\|\fr_N\|.
        \end{align*}

        \noindent\textit{Step 2.} Let us define
            \[u^1(t):=u(t)-\fu-W_{N,h}(t),\]
            where $W_{N,h}(t):=t\PPPPP_N h+W(t)$ with $W(t)$ being the Wiener process given in \eqref{GEAP4}. Then,~$u^1$ solves the equation 
            $$
                \partial_t u^1-\nu\Delta u^1+a u^1+ \Pi(u \cdot\nabla) u- \Pi(\fu \cdot\nabla) \fu=\nu\Delta W_{N,h}-aW_{N,h}.
           $$
            Testing this equation against $u^1$, we get 
            \begin{align}
            \label{R4}
            \frac{1}{2}&\frac{\mathrm{d}}{\mathrm{d}t}\|u^1\|^2+a\|u^1\|^2+\nu\|\nabla u^1\|^2\notag\\&=\langle \Pi(\fu \cdot\nabla) \fu-\Pi(u \cdot\nabla) u , u^1\rangle+\langle\nu\Delta W_{N,h}-aW_{N,h},u^1\rangle=:I_4+I_5.
            \end{align}
            By the cancellation property of the convection term and the Ladyzhenskaya inequality, it follows that
            \begin{align}\label{R5}
                I_4&=-\langle\Pi (u\cdot\nabla) (u^1+W_{N,h}),u^1\rangle-\langle\Pi ((u^1+W_{N,h})\cdot\nabla) \fu,u^1\rangle\notag\\&=-\langle(u^1\cdot\nabla) W_{N,h},u^1\rangle\notag\\&\qquad\qquad\qquad-\langle ((u^1+W_{N,h})\cdot\nabla) \fu,u^1\rangle-\langle((\fu+W_{N,h})\cdot\nabla) W_{N,h},u^1\rangle\notag\\&\lesssim \left(\|\nabla W_{N,h}\|+\|\nabla \fu\|\right)\|u^1\|\|u^1\|_{H^1}\notag\\&\qquad\qquad\qquad+ \left(\|\fu\|_{H^1}\|W_{N,h}\|_{H^1}+\|W_{N,h}\|_{H^1}^2\right)\|u^1\|_{H^1},
            \end{align}
            and
            \begin{align}
                \label{R6}
                I_5\lesssim_{a,\nu}\|W_{N,h}\|_{H^1}\|u^1\|_{H^1}.
            \end{align}
            Combining \eqref{R4}--\eqref{R6}, we obtain
            \begin{align}
                \label{R8}
                \frac{\mathrm{d}}{\mathrm{d}t}\|u^1\|^2+a\|u^1\|^2+\nu\|\nabla u^1\|^2&\le C_{a,\nu}\left(\|W_{N,h}\|_{H^1}^2+\|\fu\|_{H^1}^2\right)\|u^1\|^2\notag\\&\quad+C_{a,\nu}\left(\|W_{N,h}\|^2_{H^1}+\|W_{N,h}\|^4_{H^1}\right)\notag\\&\quad +C_{a,\nu}\|\fu\|_{H^1}^2\|W_{N,h}\|^2_{H^1}.
            \end{align}
            Choose $N_0$ sufficiently large so that 
            \begin{align}
                \label{R8-1}C_{a,\nu}\|\fu\|^2_{H^1}\le C_{a,\nu}\|\fr_N\|^2\le \frac{a}{4},\qquad \|\fu\|\le 1
            \end{align}
            for any $N\ge N_0$. For $\delta\in(0,1)$, we introduce
            \[\Omega_{T,\delta}:=\left\{\sup_{t\in[0,T]}\|W_{N,h}\|_{H^1}\le \delta\right\}.\]
            Then, one has 
            \begin{align}
                \label{R9}
                \frac{\mathrm{d}}{\mathrm{d}t}\|u^1\|^2+\frac{3a}{4}\|u^1\|^2\le C_{a,\nu} \delta\|u^1\|^2+C_{a,\nu} \delta
            \end{align}
            for any $t\le T$ on $\Omega_{T,\delta}$. Taking $\delta$ sufficiently small to absorb the first term on the right-hand side of \eqref{R9} into the left-hand side and applying the Gronwall inequality, we obtain
            \begin{align}
                \label{R10}
                \|u^1(T)\|^2\le C_{a,\nu}\delta+e^{-\frac{aT}{2}}\|u_0-\fu\|^2\le C_{a,\nu}\delta+e^{-\frac{aT}{2}}(R+1)^2.
            \end{align}
            Finally, given $d>0$, choosing $T_0$ sufficiently large and shrinking $\delta$ if necessary, we get for any $T\ge T_0$ 
            \[\|u(T)-\fu\|\le \|u^1(T)\|+\|W_{N,h}(T)\|<d\]
            on $\Omega_{T,\delta}$. Notice that by the Cameron--Martin theorem and the symmetry of the centered small-ball event, it follows
            \begin{align*}
                \mathbb{P}(\Omega_{T,\delta})\ge \exp\left(-\frac{T\BB_N^*\|h\|^2}{2}\right)\mathbb{P}\left\{\sup_{t\in[0,T]}\|W(t)\|_{H^1}\le \delta\right\}>0,
            \end{align*}
            provided that \eqref{GEAP0} holds. This leads to the desired estimate \eqref{R1}.}
        \end{proof}
        
      Combining Lemma \ref{lemma2} with Propositions \ref{proposition4} and \ref{proposition5} and using the coupling construction of the extension
    $(\bm{u}(t),\mathbb{P}_{\bm{u}})$, one can deduce that this extension is also irreducible. As the argument mirrors the proof for the complex Ginzburg--Landau equation case (cf. Proposition 4.2 in \cite{NZ24}), we omit the details.
     
        {\begin{lemma}\label{lemma3} 
         There is an integer $N_0\ge 1$ depending on $a,\nu,h,\BB_\varphi$ such that if \eqref{GEAP0} is satisfied for some $N\ge N_0$, then there exists a function $\fu\in H^1$ with $\|\fu\|\le 1$ depending on $a,\nu,h,\BB_\varphi,N$, such that the following assertion holds. For any $R,d>0$, there is $T_0>0$ such that for every $T\ge T_0$,
        \begin{align}\label{R11}
        \inf_{\bm{u}\in B_H(0,R)\times B_H(0,R)}\mathbb{P}_{\bm{u}}\{\bm{u}(T)\in B_H(\fu,d)\times B_H(\fu,d)\}\ge p>0.
        \end{align}
        Here, the constant $T_0$ depends on $R,d,a,\nu,h,N,\{b_j\}_{j\ge1}$, while $p$ additionally depends on $T$.
    	\end{lemma}
        }
        
        From the estimate \eqref{ME1}, it follows that
        $F(u):=1+\|u\|^2$
        is a Lyapunov function
        for the family $(u(t),\mathbb{P}_u)$.~Combining this and the irreducibility property in the previous lemma, we obtain the recurrence property \eqref{MR2}, as explained in Section 3 in \cite{S08}. To summarize, we have the following result.
        \begin{proposition}\label{proposition7}
        { Let $N_0$ be given in Lemma~\ref{lemma3}, fix $N\ge N_0$ satisfying \eqref{GEAP0}, and let $\fu$ be the corresponding function in that lemma. For any $d>0$, there is a constant $T_{\mathrm{rec}}>0$ such that the following assertion is true. For every $T\ge T_{\mathrm{rec}}$, there are constants $\delta_T,C_T>0$ satisfying
        \[\mathbb{E}_{\bm{u}}\exp(\delta_T\tau_d)\le C_T\left(1+\|u\|^2+\|u'\|^2\right)\]
        for any $\bm{u}\in H\times H$, where 
        \[\tau_d:=\inf\left\{k\ge 0:\bm{u}(kT)\in\overline B_H(\fu,d)\times\overline B_H(\fu,d)\right\}.\]
        The constant $T_{\mathrm{rec}}$ depends on $d,a,\nu,h,N,\{b_j\}_{j\ge1}$, while $\delta_T$ and $C_T$ additionally depend on $T$.}
        \end{proposition}

        \subsection{Polynomial squeezing}\label{squeezing}

       In this subsection, we show that the polynomial squeezing property is verified for the extension $(\bm{u}(t),\mathbb{P}_{\bm{u}})$. We continue to denote by
       $u,u',v,\tilde{u},\tilde{u}',\tilde{v}$ the processes constructed in Section~\ref{MR} and by $\hat{u},\hat{u}',\hat{v}$ the truncated processes appearing in the proof of Proposition~\ref{proposition6}. 
        The constants $K,\gamma,\mathscr{C}$ are such that Corollary~\ref{corollary1} holds. 
    
        The stopping time $\sigma$ in Theorem \ref{theorem2} is defined by 
    	$\sigma:=\tilde{\tau}\wedge\sigma_1,$
    	where 
    	$
    	\tilde{\tau}:=\tau^{\tilde{u}}\wedge\tau^{\tilde{u}'}
    	$ and 
    	$$
    	\sigma_1:=\inf\{t\ge0: \tilde{u}'(t)\neq \tilde{v}(t)\}.
    	$$ 
        Let $T>0$ be a time step to be fixed at the end of this subsection and consider the events  
        \begin{align*}
        	 \mathcal{Q}_k'&:=\{kT\le \sigma< (k+1)T,\,\sigma_1\ge \tilde{\tau}\},\\\mathcal{Q}_k^{''}&:=\{kT\le \sigma<(k+1)T,\,\sigma_1<\tilde{\tau}\}, \quad k\ge0.
        \end{align*} 
        Before verifying the polynomial squeezing property for $\sigma$, we prove the following lemma.
        {\begin{lemma}\label{lemma4} 
        There exists an integer $N_0\ge1$ depending on $a,\nu,h,\mathcal{B}_{\varphi}$ and a universal constant $L>0$ such that if \eqref{GEAP0} holds for some $N\ge N_0$, then the following statement is true. For any $q>1$, there are constants $T_{\mathrm{sq}},\rho>0$ depending on $q,a,\nu,h,\mathcal{B}_{\varphi}$ and a constant $\fd\in (0,1]$ depending additionally on $\BB^*_N,D_N$ such that for any $T\ge T_{\mathrm{sq}}$, $k\ge 0$, and 
        \[\bm{u}\in\overline{B}_{H}(0,2)\times \overline{B}_{H}(0,2),\qquad \|u-u'\|\le 2\fd,\]
        it follows that
        \[\mathbb{P}_{\bm{u}}(\mathcal{Q}_k')\vee \mathbb{P}_{\bm{u}}(\mathcal{Q}_k^{''})\le \frac{q-1}{16q}(k+1)^{-4q}.\]
        \end{lemma}}
        \begin{proof}
            {\it Step 1: estimate for $\mathbb{P}_{\bm{u}}(\mathcal{Q}_k')$.} Applying Corollary \ref{corollary1}, we have 
            \begin{align*}
                \mathbb{P}_{\bm{u}}(\mathcal{Q}_k')\le \mathbb{P}_{\bm{u}}\{kT\le \tilde{\tau}< \infty\}\lesssim_{q,a,\nu,h,\mathcal{B}_{\varphi}}\frac{1}{(\rho+LkT)^{4q}}, 
            \end{align*}
            where by choosing $L=1$, $\rho=T_1$, and $T\ge T_1$ with sufficiently large $T_1> 2$ depending on $q,a,\nu,h,\mathcal{B}_{\varphi}$, we obtain
            \begin{align}\label{PS1}
                \mathbb{P}_{\bm{u}}(\mathcal{Q}_k')\le \frac{q-1}{16q}(k+1)^{-4q}, \quad k\ge 0.
            \end{align}

            \noindent{\it Step 2: estimate for $\mathbb{P}_{\bm{u}}(\mathcal{Q}_0^{''})$.} Using the fact that $\{\tilde{v}(t)\}_{t\in[0,T]}$ and $\{\tilde{u}'(t)\}_{t\in[0,T]}$ are flows of the maximal coupling $(\mathcal{V}_T(u,u'),\mathcal{V}_T'(u,u'))$, we get
            \begin{align}\label{PS2}
            \mathbb{P}_{\bm{u}}(\mathcal{Q}_0^{''})&\le \mathbb{P}_{\bm{u}}\{0\le \sigma_1<T\}=\mathbb{P}_{\bm{u}}\{\tilde{u}'(t)\neq \tilde{v}(t)\ \mbox{for some $t\in[0,T]$}\}\notag\\
            &=\|\lambda_T(u,u')-\lambda_T'(u,u')\|_{\textup {var}},
            \end{align}
            where $\lambda_T(u,u')$ and $\lambda'_T(u,u')$ are the distributions of $\{v(t)\}_{t\in[0,T]}$, $\{u'(t)\}_{t\in[0,T]}$. The term on the right-hand side of \eqref{PS2} is estimated as follows
        \begin{align}\label{PS3}
            \|\lambda_T(u,u')&-\lambda_T'(u,u')\|_{\textup {var}}\notag\\&=\sup_{\Gamma\in  \mathscr{B}(C([0,T];H))}|\mathbb{P}\{v(\cdot)\in \Gamma\}-\mathbb{P}\{u'(\cdot)\in \Gamma\}|\notag\\&\le \mathbb{P}\{\tau^v<\infty\}+\mathbb{P}\{\tau^{u'}<\infty\}\notag\\&\quad+\sup_{\Gamma\in\mathscr{B}(C([0,T];H))}|\mathbb{P}\{v(\cdot)\in \Gamma,\tau^v=\infty\}-\mathbb{P}\{u'(\cdot)\in \Gamma,\tau^{u'}=\infty\}|\notag\\&=:\mathcal{P}_1+\mathcal{P}_2+\mathcal{P}_3,
        \end{align}
        where the term $\mathcal{P}_3$ is estimated by 
        \begin{align}\label{PS5}
            \mathcal{P}_3&=\sup_{\Gamma\in\mathscr{B}(C([0,T];H))}|\mathbb{P}\{\hat v(\cdot)\in \Gamma,\tau^{\hat v}=\infty\}-\mathbb{P}\{\hat u'(\cdot)\in \Gamma,\tau^{\hat u'}=\infty\}|\notag\\&\le \|\mathbb{P}-\Phi_*^{u,u'}\mathbb{P}\|_{\textup {var}}.
        \end{align}
        Here, we used the definitions of $\hat{u}'$, $\hat{v}$ and the equality \eqref{GEAP6}, and   $\Phi^{u,u'}$ is the transformation given in \eqref{GEAP5-1}. Let $\tilde\Phi^{u,u'}$ be the auxiliary transformation defined in \eqref{GEAP5}. By combining \eqref{PS2}--\eqref{PS5} with \eqref{GEAP7-1}, one has 
        \[\mathbb{P}_{\bm{u}}(\mathcal{Q}_0^{''})\le \|\mathbb{P}-\tilde\Phi_*^{u,u'}\mathbb{P}\|_{\textup {var}}+\mathbb{P}\{\tau^{u}<\infty\}+\mathbb{P}\{\tau^{u'}<\infty\}+\mathbb{P}\{\tau^{v}<\infty\}.\]After enlarging $T_1$ if necessary, by applying Proposition~\ref{proposition6}, Corollary~\ref{corollary1}, and \eqref{GEAP25-1} with $d=2$ and $\fd$ replaced by $2\fd$, one has
        \begin{align}
            \label{PS6}
            \mathbb{P}_{\bm{u}}(\mathcal{Q}_0^{''})\le\left(\exp\left(C\BB^*_ND_N^2\fd^2e^{C\rho}\right)-1\right)^{\frac{1}{2}}+ \frac{q-1}{32q},
        \end{align}
        where $C$ is a constant depending on $a,\nu,h,\mathcal{B}_{\varphi}$. Choosing $\fd\le 1$ sufficiently small such that 
        \begin{align}
            \label{PS6-1}\left(\exp\left(C\BB^*_ND_N^2\fd^2e^{C\rho}\right)-1\right)^{\frac{1}{2}}\le \frac{q-1}{32q},
        \end{align}
        we conclude that 
        \begin{align}
            \label{PS6-2}\mathbb{P}_{\bm{u}}(\mathcal{Q}_0^{''})\le \frac{q-1}{16q}
        \end{align}
        as desired.
    
        \noindent{\it Step 3: estimate for $\mathbb{P}_{\bm{u}}(\mathcal{Q}_k^{''})$, $k\ge1$.} In the case $k\ge1$, the estimate is proved by using the Markov property and the arguments used in Step 2. We begin by applying the Markov property:
        \begin{align*} 
        \mathbb{P}_{\bm{u}}(\mathcal{Q}_{k}^{''})&=\mathbb{P}_{\bm{u}}(\mathcal{Q}_{k}^{''},\sigma\ge kT)=\mathbb{E}_{\bm{u}}\left(\mathbb{I}_{\{\sigma\ge kT\}}\mathbb{E}_{\bm{u}}\left(\mathbb{I}_{\mathcal{Q}_{k}^{''}}\,|\,\mathscr{F}_{kT}\right)\right)\notag\\&\le \mathbb{E}_{\bm{u}}\left(\mathbb{I}_{\{\sigma\ge kT\}}\mathbb{P}_{\bm{u}(kT)}\{0\le \sigma_1< T\}\right).
        \end{align*}
        where, by proceeding similarly to Step 2, we have
        \begin{align*} 
        \mathbb{P}_{\bm{u}(kT)}\{0\le \sigma_1< T\}&\le \mathbb{P}_{\bm{u}(kT)}\{\tau^{u}<\infty\}+\mathbb{P}_{\bm{u}(kT)}\{\tau^{u'}<\infty\}\nonumber\\&\quad+\mathbb{P}_{\bm{u}(kT)}\{\tau^{v}<\infty\}+\|\mathbb{P}-\tilde\Phi_*^{\tilde{u}(kT),\tilde{u}'(kT)}\mathbb{P}\|_{\textup {var}}.
        \end{align*} 
        Here, with a slight abuse of notation, we denote by $\tau^u,\tau^{u'},\tau^v$ the stopping times that are defined in \eqref{G47}, but with the parameter $\rho$ replaced by $\rho+\epsilon' akT$ for some $\epsilon'$ to be specified later. Combining this with \eqref{GEAP7} and \eqref{GEAP7-1}, we see that
        \begin{align*}
        \mathbb{P}_{\bm{u}(kT)}\{0\le \sigma_1<T\}&\le 2\mathbb{P}_{\bm{u}(kT)}\{\tau^{u}<\infty\}+2\mathbb{P}_{\bm{u}(kT)}\{\tau^{u'}<\infty\}\\&\qquad+2\|\mathbb{P}-\tilde\Phi_*^{\tilde{u}(kT),\tilde{u}'(kT)}\mathbb{P}\|_{\textup {var}}.
        \end{align*}
        Therefore, it follows that
        \begin{align}\label{PS7}
        \mathbb{P}_{\bm{u}}(\mathcal{Q}_k^{''})&\le 2\mathbb{E}_{\bm{u}}\left(\mathbb{I}_{\{\sigma\ge kT\}}\|\mathbb{P}-\tilde\Phi_*^{\tilde{u}(kT),\tilde{u}'(kT)}\mathbb{P}\|_{\textup {var}}\right)\notag\\&\quad+2\mathbb{E}_{\bm{u}}\left(\mathbb{I}_{\{\sigma\ge kT\}}\mathbb{P}_{\bm{u}(kT)}\{\tau^{u}<\infty\}\right)\notag\\&\quad+2\mathbb{E}_{\bm{u}}\left(\mathbb{I}_{\{\sigma\ge kT\}}\mathbb{P}_{\bm{u}(kT)}\{\tau^{u'}<\infty\}\right)\notag\\&=:2I^1+2I^2+2I^3.
        \end{align}
        To bound the term $I^1$, we proceed as in the proof of Proposition \ref{proposition6}. Note that 
        \begin{align*}
            &\tilde\Phi^{\tilde{u}(kT),\tilde{u}'(kT)}(\omega)_t\\&\qquad\qquad=\omega_t-\int_0^t\mathbb{I}_{\{s\le \tau^{k}\}}\mathrm{P}_N[\Pi(u_k\cdot\nabla) u_k-\Pi(v_k\cdot\nabla)v_k -\nu\Delta(u_k-v_k)]\mathrm{d}s,
        \end{align*}
        where $u_{k},u'_k, v_k$ are the solutions of \eqref{INTRO1} and \eqref{MR7} issued from $\tilde u(kT), \tilde u'(kT)$,  and $\tilde u'(kT)$, and $\tau^k:=\tau^{u_k}\wedge\tau^{v_k}$. In view of \eqref{GEAP9} and \eqref{GEAP10}, we need to study the term  $\int_0^{\infty}\|\mathcal{A}_k(t)\|^2\mathrm{d}t$, where
        \[\mathcal{A}_k(t):=-\mathbb{I}_{\{t\le \tau^{k}\}}\mathrm{P}_N[\Pi(u_k\cdot\nabla) u_k-\Pi(v_k\cdot\nabla)v_k-\nu\Delta(u_k-v_k)].\]
        As before, we need to get a pathwise estimate for $g_k(t):=u_k(t)-v_k(t)$ before~$\tau^k$. Applying Proposition \ref{proposition5} with 
        \begin{align}
            \label{PS8}
            \epsilon:=\frac{a}{4C_1(K+2L)} \left(1\wedge \frac{1}{4\mathscr{C}}\right),
        \end{align}
        where $C_1$ is the constant  in Proposition \ref{proposition5}, we find $T_2\ge2$ and $N_0\ge 1$ depending on $a,\nu,h,\mathcal{B}_{\varphi}$ such that \eqref{STA2} holds for any $N\ge N_0$. We enlarge $N_0$, if necessary, so that Proposition~\ref{proposition6} also applies with $L=1$. For any $t\ge 0$, we write 
        \[\fg_{k,t}:=\|g_k(t)\|^2+\|\PPPPP_N g_k(t)\|^2_{H^2}.\]
        If $\tau^k\le T_2$, then by proceeding similarly to the derivation of \eqref{GEAP14},   
        \begin{align}\label{PS9}
        \|g_k(t)\|^2 \le C\fg_{k,0}\exp\left(-at+C_{a,\nu}\int_0^t\left(\|u_k\|^2_{H^1}+\|v_k\|^2_{H^1}\right)\mathrm{d}s\right),
        \end{align}
        where $C>0$ is a constant depending on $a,\nu,h,\mathcal{B}_{\varphi}$. Notice that by the definition of $\tau^k$,
        \begin{align}
        \label{PS10}
        \mathcal{E}_{1}^{u_k}(t)&< (K+L)t+\rho+\epsilon' akT+\mathscr{C}\|u_k(0)\|^2,\\\label{PS10-1}  \mathcal{E}_{1}^{v_k}(t)&< (K+L)t+\rho+\epsilon' akT+\mathscr{C}\|v_k(0)\|^2
    	\end{align}
        for any $t< \tau^k$. Then, from \eqref{PS9}--\eqref{PS10-1}, and the Young inequality, it follows
        \begin{align}\label{PS11}
        \notag\|g_k(t)\|^2 &\le C\fg_{k,0} \exp\left(-at+C[(K+L)T_2+\rho+\epsilon' akT+\|u_k(0)\|^2+\|v_k(0)\|^2]\right)\notag\\&\le C\fg_{k,0}\exp\left(-at+\frac{a(\|\tilde{u}(kT)\|^6+\|\tilde{v}(kT)\|^6)}{16(K+2L)}+C\epsilon' akT\right),
        \end{align}
        where $C$ depends on $q, a,\nu,h,\mathcal{B}_{\varphi}$. If $\tau^k>T_2$, it is clear that \eqref{PS11} holds up to~$T_2$. Therefore, by Proposition \ref{proposition5} with $\epsilon$ given by \eqref{PS8}, the definition of~$\tau^k$, \eqref{PS11}, and \eqref{STA14}, we have 
        \begin{align}
            \label{PS12}
            \|g_k(t)\|^2&\le C\fg_{k,T_2}\exp\left(-\frac{at}{2}+\frac{a(\|u_k(0)\|^6+\|v_k(0)\|^6)}{16(K+2L)}+C\epsilon' akT\right)\notag\\&\le C\fg_{k,0}\exp\left(-\frac{at}{2}+\frac{a(\|\tilde{u}(kT)\|^6+\|\tilde{v}(kT)\|^6)}{8(K+2L)}+C\epsilon' akT\right).
        \end{align}
        On the other hand, due to the definitions of $\mathcal{E}^{\tilde{u}}_{\psi}$, $\mathcal{E}^{\tilde{v}}_{\psi}$, on $\{\sigma\ge kT\}$, one has
        \begin{align}\label{PS13}
           \|\tilde{u}(kT)\|^6&\le \mathcal{E}^{\tilde{u}}_{\psi}(kT)\le (K+2L)kT+2\rho+\mathscr{C}(1+2^6),\\\label{PS14}
           \|\tilde{v}(kT)\|^6&\le \mathcal{E}^{\tilde{v}}_{\psi}(kT)\le (K+2L)kT+2\rho+\mathscr{C}(1+2^6),
        \end{align}
        and 
         \begin{align}
            \label{PS15}
            \fg_{k,0}&=\|\tilde{u}(kT)-\tilde{v}(kT)\|^2+\|\PPPPP_N (\tilde{u}(kT)-\tilde{v}(kT))\|^2_{H^2}\notag\\&\le CD_N\fd^2\exp\left(-\frac{a}{2}kT+C\rho\right).
        \end{align}
        Hence, by combining the estimates \eqref{PS11}--\eqref{PS15}, we infer
        \begin{align}\label{PS16-1}
            \|g_k(t)\|^2\le CD_N\fd^2\exp\left(-\left(\frac{1}{4}-C\epsilon'\right)akT-\frac{a}{2}t\right)
        \end{align}
        for any $t\in[0,\tau^k)$ with the constant $C>0$ depending on $q,a,\nu,h,\mathcal{B}_{\varphi}$. Choosing $\epsilon'=\frac{1}{8C}\wedge 1$, we get
        \begin{align}\label{PS16}
            \|g_k(t)\|^2\le CD_N\fd^2\exp\left(-\frac{a}{8}kT-\frac{a}{2}t\right).
        \end{align}
        Moreover, as in \eqref{GEAP22},
        \begin{align}\label{PS17}
            \|\mathcal{A}_k(t)\|^2\le CD_N\mathbb{I}_{\{t\le \tau^k\}}\|g_k\|^2\left(1+\|u_k\|^2_{H^1}+\|v_k\|^2_{H^1}\right),
        \end{align}
        where by \eqref{PS10}, \eqref{PS10-1}, \eqref{PS13}, \eqref{PS14}, 
        \begin{align}\label{PS18}
            \int_0^t(\|u_k\|^2_{H^1}+&\|v_k\|^2_{H^1})\mathrm{d}s\le \mathcal{E}^{u_k}_1(t)+\mathcal{E}^{v_k}_1(t)\notag\\&\le 2(K+L)t+2\rho+2\epsilon' akT+\mathscr{C}(\|\tilde{u}(kT)\|^2+\|\tilde{v}(kT)\|^2)\notag\\&\le 2(K+L)t+2\rho+2akT+\mathscr{C}(\mathcal{E}^{\tilde{u}}_{\psi}(kT)+\mathcal{E}^{\tilde{v}}_{\psi}(kT))\notag\\&\le 2(K+L)t+2\rho+2  akT+\mathscr{C}[2(K+2L)kT+4\rho+2\mathscr{C}(1+2^6)]\notag\\&\le C(t+kT+1)+C\le C(1+kT)(1+t)
        \end{align}
        for $t<\tau^k$. Proceeding similarly to the derivation of the estimates \eqref{GEAP23} and~\eqref{GEAP24} combined with an application of \eqref{PS17} and~\eqref{PS18}, we get
        \begin{align}
            \label{PS19}
            \int_0^{\infty}\|\mathcal{A}_k(t)\|^2\mathrm{d} t\le CD_N^2\fd^2e^{-\frac{akT}{16}}
        \end{align}
        on the event $\{\sigma\ge kT\}$. Hence, 
        \begin{align*}
        I^1\le \left(\exp\left(C\BB^*_ND_N^2\fd^2e^{-\frac{akT}{16}}\right)-1\right)^{\frac{1}{2}}.
        \end{align*}
        We choose $\fd<1$ so small that
        \begin{align}
            \label{PS19-1}C_\fd:=C\BB^*_ND_N^2\fd^2\le 1.
        \end{align}
        Note that there is a constant $c>0$ such that $e^x-1\le cx$ for any $x\in(0,1)$. Then, by choosing $T\ge\frac{32}{a}$, we derive 
        \begin{align*}
            I^1\lesssim \sqrt{C_\fd}e^{-k}\lesssim_q\sqrt{C_\fd}(k+1)^{-4q},
        \end{align*}
        which implies 
        \begin{align}
            \label{PS20}
            I^1\le \frac{q-1}{96q}(k+1)^{-4q},
        \end{align}
        provided that $\fd$ is sufficiently small. Next, we bound the terms $I^2, I^3$ in~\eqref{PS7}. For $I^2$, we apply Corollary \ref{corollary1} and \eqref{ME1} to derive
        \begin{align}\label{PS21}
            I^2&=\mathbb{E}_{\bm{u}}\left(\mathbb{I}_{\{\sigma\ge kT\}}\mathbb{P}_{\bm{u}(kT)}\{\tau^{u}<\infty\}\right)\le \mathbb{E}_{\bm{u}}\left(\mathbb{P}_{\bm{u}(kT)}\{\tau^{u}<\infty\}\right)\notag\\&\le C \left(1+\E_{\bm{u}}\|\tilde u(kT)\|^{8(8q+3)}\right)\left(e^{-\gamma(\rho+\epsilon' akT)}+\frac{1}{(\rho+\epsilon' akT)^{4q}}\right)\notag\\&\lesssim_{q,a,\nu,h,\BB_{\varphi}}\frac{1}{(\rho+\epsilon' akT)^{4q}}.
        \end{align}
        By proceeding similarly to the derivation of \eqref{PS1}, there is $T_3$ depending on $q,a,\nu,h,\mathcal{B}_{\varphi}$ such that if $T\ge T_3$, then
        \begin{align}
            \label{PS22}
            I^2\le \frac{q-1}{96q}(k+1)^{-4q}.
        \end{align}
        The term $I^3$ is estimated similarly, with the help of Corollary \ref{corollary1}. Consequently, it follows
        \begin{align}
            \label{PS23}\mathbb{P}_{\bm{u}}(\mathcal{Q}_k^{''})\le  \frac{q-1}{16q}(k+1)^{-4q}.
        \end{align}
    
        \noindent{\it Step 4: Conclusion.}  Let $L=1$ and $N_0$ be chosen in Step 3. Suppose that \eqref{GEAP0} is satisfied for some $N\ge N_0$. For any $q>1$, we take $\rho=T_1$ with $T_1$ given in Step 1 and choose $\fd\in(0,1]$ sufficiently small so that \eqref{PS6-1}, \eqref{PS19-1}, and \eqref{PS20} hold. Let
        \[T_{\mathrm{sq}}=\max\left\{T_1,T_2,T_3,\frac{32}{a}\right\},\]
        where $T_2,T_3$ are chosen in Step 3. By combining \eqref{PS1}, \eqref{PS6-2}, and \eqref{PS23}, we conclude the proof.
        \end{proof}
    {Let $N_0 \geq 1$ be sufficiently large so that both Lemma~\ref{lemma4} and Proposition~\ref{proposition7} are applicable with $N= N_0$. Suppose that \eqref{INTRO8} is satisfied. Fix $q>1$, and let $\fd$, $\rho$, and $T_{\mathrm{sq}}$ be given by Lemma~\ref{lemma4}. Let
$T_{\mathrm{rec}}=T_{\mathrm{rec}}(\fd)$
    be the threshold in Proposition~\ref{proposition7}. We now fix
    \[T\ge \max\{T_{\mathrm{sq}},T_{\mathrm{rec}}\}\]
    and construct the extension $(\bm{u}_t,\mathbb{P}_{\bm{u}})$ with this time step, so that the recurrence and polynomial squeezing properties hold simultaneously for the same $N_0$ and $T$.}

        We turn to the verification of properties \eqref{MR3}--\eqref{MR6}. Let $B:=\overline{B}_H(\fu,\fd)$ with $\fu$ given in Proposition~\ref{proposition7}. The property \eqref{MR3} follows from the definition of $\sigma$, \eqref{GEAP18}, and \eqref{GEAP19}. Applying Lemma \ref{lemma4}, we obtain 
        \[\mathbb{P}_{\bm{u}}\{\sigma=\infty\}\ge 1-\sum_{k=0}^{\infty}\mathbb{P}_{\bm{u}}\{\sigma\in[kT,(k+1)T)\}\ge \frac{1}{2}>0,\]
        which establishes \eqref{MR4}. To prove \eqref{MR5}, we derive from Lemma \ref{lemma4} that
        \begin{align*}
            \mathbb{E}_{\bm{u}}\left(\mathbb{I}_{\{\sigma<\infty\}}\sigma^q\right)&\le \sum_{k=0}^{\infty}\mathbb{E}_{\bm{u}}\left(\mathbb{I}_{\{\sigma\in[kT,(k+1)T)\}}\sigma^q\right)\\&\lesssim_{q,T} \sum_{k=0}^{\infty}(k+1)^q\mathbb{P}_{\bm{u}}\{\sigma\in[kT,(k+1)T)\}\\&\lesssim_{q,T} \sum_{k=0}^{\infty}(k+1)^{-3q}\le C_q.
        \end{align*}
        Lastly, to check \eqref{MR6}, using the definition of $\sigma$, we note that
        \begin{align*}
        \mathbb{E}_{\bm{u}}\left(\mathbb{I}_{\{\sigma<\infty\}}\left(\|\tilde{u}(\sigma)\|^{2q}+\|\tilde{u}'(\sigma)\|^{2q}\right)\right)&\le C_q\mathbb{E}_{\bm{u}}(\mathbb{I}_{\{\sigma<\infty\}}(1+\sigma^{q}))\le C_q.
        \end{align*} 
        This completes the proof of polynomial squeezing, thereby completing the proof of Theorem~\ref{theorem2}.
         
        \appendix
        \section{Moment estimates}
        The following lemma gathers standard a priori estimates for the stochastic NS system \eqref{INTRO1}; see Chapter~2 in \cite{KS12} for analogous results in the case of bounded domains.
        \begin{lemma}\label{lemmaA1}
            Let $u(t)$ be a solution of \eqref{INTRO1}, and let $w(t)$ be the corresponding vorticity. Then, for any $p,p'\ge1$ and $T>0$, we have 
            \begin{align}
                \mathbb{E}\|u(T)\|^{2p}&\le e^{-a pT} \mathbb{E}\|u_0\|^{2p}+C_{a,p,h,\mathcal{B}_\varphi},\label{ME1}\\
                \mathbb{E}(\mathcal{E}^u_p(T))^{p'}&\le C_{a,\nu,h,\mathcal{B}_\varphi,p,p'}\left( \mathbb{E}\|u_0\|^{4p(p'+2)}+T^{p'}+1\right)\label{ME2},\\
                \mathbb{E}\|w(T)\|^{2p}&\le C_{a,\nu,h,\mathcal{B}_{\varphi},p,T}\left(\mathbb{E}\|u_0\|^{4(p+2)}+1\right),\label{ME3}
            \end{align}
            where $\mathcal{E}^u_p$ is defined in \eqref{G2}.
        \end{lemma} 
        \begin{proof}
            First, let us establish the estimate \eqref{ME1}. Taking the expectation in \eqref{G5} and using the Young inequality, we get 
            $$
                \frac{\mathrm{d}}{\mathrm{d}t}\mathbb{E}\|u(t)\|^{2p}+2p\nu\mathbb{E}\|\nabla u(t)\|^2\|u(t)\|^{2p-2}+ap \mathbb{E}\|u(t)\|^{2p}\le C_{a,p,h,\mathcal{B}_\varphi}.
        $$
            Applying the Gronwall inequality, we arrive at the required estimate. 
    
            Next, we turn to the estimate \eqref{ME2}. Note that if $X_1$ and $X_2$ are non-negative random variables, then
            \begin{align}\label{ME5}
                \mathbb{E}X_1^{p'}&\le 2^{p'} \mathbb{E}((X_1-X_2)^{p'}\mathbb{I}_{\{X_1\ge X_2\}})+2^{p'} \mathbb{E}X_2^{p'}\notag\\&\le 2^{p'}\int_0^{\infty}\mathbb{P}\{X_1-X_2\ge \rho^{\frac{1}{p'}}\}\mathrm{d }\rho+2^{p'} \mathbb{E}X_2^{p'}. 
            \end{align}
            Applying this inequality with 
            \[X_1:=\mathcal{E}^u_p(T),\qquad X_2:=\kappa_p T+\mathcal{C}_p\|u_0\|^{2p},\] 
            where $\kappa_p,\mathcal{C}_p$ are given in Proposition \ref{proposition1}, we get
            \begin{align*}
                \mathbb{E}(\mathcal{E}^u_p(T))^{p'}&\le 2^{p'}\int_0^{\infty}\mathbb{P}\left\{\sup_{t\ge 0}\left(\mathcal{E}^u_p(t)-\kappa_p t-\mathcal{C}_p\|u_0\|^{2p}\right)\ge \rho^{\frac{1}{p'}}\right\}\mathrm{d}\rho\\&\quad+4^{p'}\left(\mathcal{C}_p^{p'}\mathbb{E}\|u_0\|^{2pp'}+\kappa_p^{p'}T^{p'}\right).
            \end{align*}
            By choosing $q:=2(p'+2)$ in Proposition \ref{proposition1}, we obtain
            \begin{align*}
                \mathbb{E}(\mathcal{E}^u_p(T))^{p'}&\le 2^{p'} C_{a,\nu,h,\mathcal{B}_\varphi,p,p'}\left(\mathbb{E}\|u_0\|^{(2p-1)(2p'+4)}+1\right)\int_2^{\infty}\frac{1}{\rho^{1+\frac{1}{p'}}}\mathrm{d}\rho\\&\quad+2^{p'+1}+4^{p'}\left(\mathcal{C}_p^{p'}\mathbb{E}\|u_0\|^{2pp'}+\kappa_p^{p'}T^{p'}\right)\\& \le C_{a,\nu,h,\mathcal{B}_\varphi,p,p'}\left( \mathbb{E}\|u_0\|^{4p(p'+2)}+T^{p'}+1\right)
            \end{align*}
            as desired. Finally, to prove \eqref{ME3}, we apply the It\^o formula and use \eqref{G11_0}: 
            \begin{align*} 
            \mathrm{d}(t\|w(t)\|^{2})&=2t\langle w,\nu\Delta w-aw+\curl h\rangle\mathrm{d}t+\mathrm{d}\tilde{M}_{1}(t)\notag\\&\qquad\qquad\qquad\qquad+t\sum_{j=1}^{\infty}b_j^2\|\curl e_j\|^2\mathrm{d}t+\|w\|^2\mathrm{d}t,
            \end{align*} 
            where
            $$
            \tilde{M}_{1}(t):=2\int_0^ts\sum_{j=1}^{\infty}b_j\langle w,\curl e_j\rangle\mathrm{d}\beta_j(s).
           $$
            Notice that for any $t\le T$, 
           $$
            \langle\tilde{M}_{1}\rangle(t):=4\int_0^ts^2\sum_{j=1}^{\infty}b_j^2\langle w,\curl e_j\rangle^2\mathrm{d}s\le  4T\mathcal{B}_{\varphi}\int_0^ts\|w(s)\|^2\mathrm{d}s.
           $$
            By setting 
               $\gamma^*:=\frac{a}{8T\mathcal{B}_{\varphi}}$
            and using the Young inequality, we obtain
            \begin{align*}
                &t\|w(t)\|^2+\int_0^ts\left(a\|w(s)\|^2+\nu\|\nabla w(s)\|^2\right)\mathrm{d}s\notag\\&\qquad\qquad\qquad\quad\le C_{a,\nu,h,\mathcal{B}_{\varphi}}t^2+\int_0^t\|w(s)\|^2\mathrm{d}s+\tilde{M}_{1}(t)-\frac{\gamma^*}{2}\langle \tilde{M}_{1}\rangle(t),
            \end{align*}
            which implies 
           $$
                t\|w(t)\|^2 \le C_{a,\nu,h,\mathcal{B}_{\varphi}}T^2+C\mathcal{E}^u_1(t)+\tilde{M}_{1}(t)-\frac{\gamma^*}{2}\langle \tilde{M}_{1}\rangle(t)
          $$
            for any $0<t\le T$. Applying the inequality \eqref{ME5} with 
            \[X_1:=t\|w(t)\|^2,\qquad X_2:=C_{a,\nu,h,\mathcal{B}_{\varphi}}T^2+C\mathcal{E}^u_1(t),\]
            the exponential supermartingale estimate, and \eqref{ME2}, we derive
            \begin{align*}
                \mathbb{E}(t^p \|w(t)\|^{2p})&\le 2^p\int_0^{\infty}\mathbb{P}\left\{\sup_{t\ge 0}\left(\tilde{M}_{1}(t)-\frac{\gamma^*}{2}\langle \tilde{M}_{1}\rangle(t)\right)\ge \rho^{\frac{1}{p}}\right\}\mathrm{d}\rho\notag\\ &\quad+C_{a,\nu,h,\mathcal{B}_{\varphi},p,T}\left(1+\mathbb{E}\left(\mathcal{E}^u_{1}(t)\right)^p\right)\notag\\&\le C_{a,\nu,h,\mathcal{B}_{\varphi},p,T}\left(1+\mathbb{E}\|u_0\|^{4(p+2)}\right).
            \end{align*}
            Setting $t=T$ in this inequality, we get the desired estimate \eqref{ME3}.
        \end{proof}

            \section{Weighted estimate for the Leray projector}
            
        In this appendix, we prove Lemma \ref{lemma6}. The proof of the property \hyperlink{(a)}{\rm(a)}  relies on weighted estimates for singular integral operators. Let us begin by recalling the following definition.
        \begin{definition} A locally integrable function $\fw:\R^n\to \R_+$ is said to be a Muckenhoupt $A_2$ weight, if 
        \[[\fw]_{A_2}:=\sup_{\text{\rm all balls $B$}}\left(\frac{1}{|B|}\int_{B}\fw\mathrm{d}x\right)\left(\frac{1}{|B|}\int_{B}\fw^{-1} \mathrm{d}x\right)<\infty,\]
        where $|B|$ denotes the Lebesgue measure of $B$.
        \end{definition}
    We will use the following weighted estimates for the Riesz transform established in \cite{Pet08}. 
        \begin{theorem}\label{theoremb1}
            There exists a constant $C>0$ such that for any $\fw\in A_2$, the Riesz transforms $\mathcal{R}_j$, $j=1,2,\ldots,n$ defined by \eqref{Riesz}
            as operators in weighted space $\mathcal{R}_j: L^2(\mathbb{R}^n,\fw\mathrm{d}x)\to L^2(\mathbb{R}^n,\fw\mathrm{d}x)$ have operator norm satisfying the estimate 
        \[\|\mathcal{R}_j\|\le C[\fw]_{A_2}.\]
        \end{theorem}
        Before proving the property \hyperlink{(a)}{\rm(a)}, let us establish a uniform-in-time bound for the quantity $[\psi^2(t)]_{A_2}$.
        {\begin{lemma}\label{lemmab1}
             Let $\psi$ be the space-time weight function given by \eqref{weight}. Then 
             \[\sup_{t\ge 1} [\psi^2 (t)]_{A_2}<\infty.\]
        \end{lemma}
        \begin{proof}
            \textit{Step 1.} In this step, we derive pointwise upper and lower bounds for $\psi$. Define 
            \[f_1(y):=1-e^{-y},\qquad f_2(y):= \frac{1-e^{-y}}{y},\qquad y>0.\]
            The function $f_1$ is strictly increasing, and $f_2$ is strictly decreasing on $(0,\infty)$. Therefore, in the ball $\{x: |x|\le \sqrt{t^4-1}\}$, we have 
            \[t\ge\varphi\ge \psi(t,x)=\varphi (1-e^{-t/\varphi})\ge \varphi(1-e^{-1}).\]
            Similarly,   
            \[\varphi\ge t\ge \psi(t,x)=t\frac{\varphi}{t}(1-e^{-t/\varphi})\ge t(1-e^{-1})\]
            in the outer region $\{x:|x|>\sqrt{t^4-1}\}$. To summarize, we have 
            $$
            (1-e^{-1})(t\wedge \varphi)\le \psi(t,x)\le t\wedge \varphi,\qquad \forall t\ge 1, x\in\mathbb{R}^2. 
            $$
            Thus, by a change of time variable $s=t^2$, it suffices to establish
            \begin{equation}
                \label{2}\sup_{s\ge 1}[s\wedge\varphi^2]_{A_2}<\infty,
            \end{equation}
            due to the definition of Muckenhoupt $A_2$ weights. For notational simplicity, we relabel $s$ as $t$ below.
    
            \noindent \textit{Step 2.} Let $\Phi:=\varphi^2$. Since 
            \[[1\wedge \varphi^2]_{A_2}=1,\]
            it remains to consider the case $t>1$. We claim that 
            \begin{equation}
                \label{3}
                \sup_{t> 1}\sup_{R>0}\left(\frac{1}{|B(0,R)|}\int_{B(0,R)}t\wedge \Phi \mathrm{d}x\right)\left(\frac{1}{|B(0,R)|}\int_{B(0,R)}(t\wedge \Phi)^{-1} \mathrm{d}x\right)<\infty.
            \end{equation}
            For $R\ge \sqrt{t^2-1}$, we compute the integral by using polar coordinates:  
            \begin{align}\label{4}
                \int_{B(0,R)}(t\wedge \Phi)^{-1} \mathrm{d}x&=2\pi \int_0^R (t\wedge \sqrt{r^2+1})^{-1}r\mathrm{d}r\notag\\&=2\pi \int_0^{\sqrt{t^2-1}}  \frac{r}{\sqrt{r^2+1}}\mathrm{d}r+2\pi \int_{\sqrt{t^2-1}}^R t^{-1}r\mathrm{d}r\notag\\&=2\pi\left((r^2+1)^{\frac{1}{2}}\big|_0^{\sqrt{t^2-1}}+\frac{r^2}{2t}\big|_{\sqrt{t^2-1}}^R\right)\notag\\&=2\pi\left(\frac{t}{2}-1+\frac{1}{2t}+\frac{R^2}{2t}\right)=\frac{\pi}{t}(R^2+(t-1)^2).
            \end{align}
            Similarly,
            \begin{align}\label{5}
                \int_{B(0,R)}t\wedge \Phi \mathrm{d}x&=2\pi \int_0^R (t\wedge \sqrt{r^2+1})r\mathrm{d}r\notag\\&=2\pi \int_0^{\sqrt{t^2-1}}  r\sqrt{r^2+1}\mathrm{d}r+2\pi \int_{\sqrt{t^2-1}}^R tr\mathrm{d}r\notag\\&=2\pi\left(\frac{1}{3}(r^2+1)^{\frac{3}{2}}\big|_0^{\sqrt{t^2-1}}+\frac{tr^2}{2}\big|_{\sqrt{t^2-1}}^R\right)\notag\\&=2\pi\left(-\frac{1}{6}t^3+\frac{t}{2}-\frac{1}{3}+\frac{tR^2}{2}\right)\notag\\&\le \pi tR^2,
            \end{align}
            where we used  
            \[-\frac{1}{6}t^3+\frac{t}{2}-\frac{1}{3}<0,\qquad\forall t>1.\]
            Plugging the estimates \eqref{4} and \eqref{5} into \eqref{3}, we obtain
            \begin{align}
                \label{6}
                \left(\frac{1}{|B(0,R)|}\int_{B(0,R)}t\wedge \Phi \mathrm{d}x\right)&\left(\frac{1}{|B(0,R)|}\int_{B(0,R)}(t\wedge \Phi)^{-1} \mathrm{d}x\right)\notag\\&\quad\le 1+\frac{(t-1)^2}{R^2}\le 1+\frac{(t-1)^2}{t^2-1}\le 2
            \end{align}
            for any $t> 1$ and $R\ge \sqrt{t^2-1}$. As for the case $R<\sqrt{t^2-1}$, we similarly have 
            \begin{align*}
                \int_{B(0,R)}(t\wedge \Phi)^{-1} \mathrm{d}x=2\pi (\sqrt{R^2+1}-1)
            \end{align*}
            and 
            \begin{align*}
                \int_{B(0,R)}(t\wedge \Phi) \mathrm{d}x=\frac{2\pi}{3} ((R^2+1)^{\frac{3}{2}}-1),
            \end{align*}
            which implies 
            \begin{align*}
                \left(\frac{1}{|B(0,R)|}\int_{B(0,R)}t\wedge \Phi \mathrm{d}x\right)&\left(\frac{1}{|B(0,R)|}\int_{B(0,R)}(t\wedge \Phi)^{-1} \mathrm{d}x\right)\notag\\&\qquad=\frac{4}{3}\frac{(\sqrt{R^2+1}-1)((R^2+1)^{\frac{3}{2}}-1)}{R^4}=:G(R).
            \end{align*}
            Noticing that
            \[\lim_{R\to\infty}G(R)=\frac{4}{3},\qquad \lim_{R\to0^+}G(R)=1,\]
             we conclude that 
            \begin{align}
                \label{7}
                \left(\frac{1}{|B(0,R)|}\int_{B(0,R)}t\wedge \Phi \mathrm{d}x\right)\left(\frac{1}{|B(0,R)|}\int_{B(0,R)}(t\wedge \Phi)^{-1} \mathrm{d}x\right)\le \sup_{R>0}G(R)\!<\infty
            \end{align}
            for any $t> 1$ and $R<\sqrt{t^2-1}$. Combining \eqref{6} and \eqref{7}, we obtain the claimed estimate \eqref{3}.
    
            \noindent\textit{Step 3.} To obtain the estimate \eqref{2}, let us divide all balls $B(x_0,R)\subset \R^2$ into two types: type I if $|x_0|\ge 3R$ and type II if $|x_0|<3R$. In this step, we derive bounds for type I balls. For any $x\in B(x_0,R)$, there holds
            \[\frac{2|x_0|}{3}\le |x_0|-R\le |x|\le |x_0|+R\le \frac{4|x_0|}{3},\]
            which implies
            \begin{align*}
                 \left(\frac{1}{|B(x_0,R)|}\int_{B(x_0,R)}t\wedge \Phi \mathrm{d}x\right)&\left(\frac{1}{|B(x_0,R)|}\int_{B(x_0,R)}(t\wedge \Phi)^{-1} \mathrm{d}x\right)\\&\quad\le\frac{t\wedge\sqrt{1+(4|x_0|/3)^2}}{t\wedge\sqrt{1+(2|x_0|/3)^2}}=:F(t,x_0).
            \end{align*}
            Furthermore,
            \begin{align}\label{8}
                F(t,x_0)=\begin{cases}
                \frac{\sqrt{1+(4|x_0|/3)^2}}{\sqrt{1+(2|x_0|/3)^2}},\qquad &\mbox{if }\sqrt{1+(4|x_0|/3)^2}\le t,\\
                \frac{t}{\sqrt{1+(2|x_0|/3)^2}},\qquad &\mbox{if }\sqrt{1+(2|x_0|/3)^2}\le t<\sqrt{1+(4|x_0|/3)^2},\\
                1,\qquad  &\mbox{if } t< \sqrt{1+(2|x_0|/3)^2}.
                \end{cases}
            \end{align}
            It is clear that 
            \begin{align}
                \label{9}
                \sup_{y>0}
                \frac{\sqrt{1+(4y/3)^2}}{\sqrt{1+(2y/3)^2}}<\infty.
            \end{align}
            Notice that if $\sqrt{1+(2|x_0|/3)^2}\le t<\sqrt{1+(4|x_0|/3)^2}$, then
            \[\frac{t^2-1}{4}\le \frac{4|x_0|^2}{9},\]
            which ensures the estimate 
            \begin{align}\label{10}
                \frac{t}{\sqrt{1+(2|x_0|/3)^2}}\le \frac{t}{\sqrt{1+\frac{t^2-1}{4}}}\le \sup_{t> 1}\frac{t}{\sqrt{1+\frac{t^2-1}{4}}}.
            \end{align}
            Combining the estimates \eqref{8}--\eqref{10}, we conclude
            \begin{align}
                \label{11}
                \sup_{t> 1}\sup_{\text{\rm Type I balls}}\left(\frac{1}{|B|}\int_{B}t\wedge \Phi \mathrm{d}x\right)\left(\frac{1}{|B|}\int_{B}(t\wedge \Phi)^{-1} \mathrm{d}x\right)<\infty.
            \end{align}
    
            \noindent \textit{Step 4.} As for type II balls, we note that 
            \[B(x_0,R)\subset B(0,4R).\]
            Therefore, we arrive at \begin{align*}
                &\left(\frac{1}{|B(x_0,R)|}\int_{B(x_0,R)}t\wedge \Phi \mathrm{d}x\right)\left(\frac{1}{|B(x_0,R)|}\int_{B(x_0,R)}(t\wedge \Phi)^{-1} \mathrm{d}x\right)\notag\\&\qquad\le \left(\frac{16}{|B(0,4R)|}\int_{B(0,4R)}t\wedge \Phi \mathrm{d}x\right)\left(\frac{16}{|B(0,4R)|}\int_{B(0,4R)}(t\wedge \Phi)^{-1} \mathrm{d}x\right)\notag \\&\qquad\lesssim \sup_{t> 1}\sup_{R>0}\left(\frac{1}{|B(0,R)|}\int_{B(0,R)}t\wedge \Phi \mathrm{d}x\right)\left(\frac{1}{|B(0,R)|}\int_{B(0,R)}(t\wedge \Phi)^{-1} \mathrm{d}x\right),
            \end{align*}
            which implies 
            \begin{align*}
                \sup_{t> 1}\sup_{\text{\rm Type II balls}}\left(\frac{1}{|B|}\int_{B}t\wedge \Phi \mathrm{d}x\right)\left(\frac{1}{|B|}\int_{B}(t\wedge \Phi)^{-1} \mathrm{d}x\right)<\infty,
            \end{align*}
            due to the estimate \eqref{3}. This, together with the estimate \eqref{11}, completes the proof.
        \end{proof}}
        Now we are in a position to prove Lemma \ref{lemma6}.
        \begin{proof}[Proof of Lemma \ref{lemma6}]An application of Theorem \ref{theoremb1} with the weight $\psi^2$ and Lemma \ref{lemmab1} implies that there exists a constant $C>0$ such~that 
        \begin{align}\label{12}\|\psi (t) \mathcal{R}_jf\|\le C \|\psi(t) f\|,\qquad \forall t\ge 1.\end{align}
        Moreover, for any vector field $f\in L^2(\R^2;\R^2)$, the Leray projector $\Pi$ can be represented as follows:
        \[(\Pi f)_i=f_i+\sum_{j=1}^2\mathcal{R}_i\mathcal{R}_jf_j.\]
        Combining this equality with the estimate \eqref{12}, we obtain the property \hyperlink{(a)}{\rm(a)}.
     
        To prove \hyperlink{(b)}{\rm(b)},  note that the assumption $\diver u=0$ implies that \[-\Delta\partial_1 u=\partial_1\nabla^{\perp}w.\]
            Taking the weighted $L^2$-inner product of this equation with $\partial_1 u$, we get
            \[-\langle \psi \partial_1 u, \psi \Delta\partial_1 u\rangle=\langle \psi \partial_1 u, \psi \partial_1\nabla^{\perp}w\rangle,\]
            where
            \[-\langle \psi \partial_1 u, \psi \Delta\partial_1 u\rangle=\|\psi \nabla \partial_1 u\|^2+\sum_{j=1}^22\langle\nabla \psi \partial_1u_j,\psi \nabla \partial_1 u_j\rangle\]
            and 
            \[\langle \psi \partial_1 u, \psi \partial_1\nabla^{\perp}w\rangle=\langle \curl(\psi^2 \partial_1 u),  \partial_1w\rangle=\|\psi \partial_1w\|^2+2\langle \nabla \psi (\partial_1u)^{\perp},\psi\partial_1 w\rangle.\]
            Here, $(\partial_1u)^{\perp}:= (\partial_1 u_2,-\partial_1 u_1)$. Therefore, 
            \[\|\psi \nabla \partial_1 u\|\lesssim \|\nabla u\|+\|\psi \nabla w\|.\]
            The weighted estimate for $\nabla\partial_2 u$ is proved in the same way.  
        \end{proof}

    \addtocontents{toc}{\protect\setcounter{tocdepth}{1}}
\addcontentsline{toc}{section}{References}
        \vspace{.2in}
    	\bibliographystyle{alpha}
    \bibliography{reference}
    \end{document}